\documentclass{amsart}

\usepackage{amsmath}
\usepackage{amssymb}
\usepackage{amsthm}
\usepackage{amsfonts,comment}
\usepackage{mathtools}
\usepackage[shortlabels]{enumitem}
\usepackage{tikz}
\usepackage{tikz-cd}
\usepackage{mleftright}
\usepackage{url}
\usepackage{graphicx}

\theoremstyle{definition}
\newtheorem{definition}{Definition}[section]
\newtheorem{example}[definition]{Example}
\newtheorem{assumption}[definition]{Assumption}

\newcounter{claim}[definition]

\newenvironment{claim*}[1][]{\refstepcounter{claim}\par\medskip
\textbf{Claim. #1} \rmfamily}{}

\newcommand*{\myproofname}{My proof}

\theoremstyle{plain}
\newtheorem{proposition}[definition]{Proposition}
\newtheorem{theorem}[definition]{Theorem}
\newtheorem{corollary}[definition]{Corollary}
\newtheorem{lemma}[definition]{Lemma}
\newtheorem{fact}[definition]{Fact}
\newtheorem{theoremnumb}{Theorem}
\newtheorem*{theorem*}{Theorem}

\theoremstyle{remark}
\newtheorem{remark}[definition]{Remark}
\newtheorem*{remark*}{Remark}

\newcommand{\id}{\textrm{\normalfont id}}
\newcommand{\diag}{\textrm{\normalfont diag}}

\newcommand{\up}{\uparrow \kern-3pt}

\newcommand{\N}{\mathbb{N}}

\newcommand{\Q}{\mathbb{Q}}
\newcommand{\R}{\mathbb{R}}
\newcommand{\C}{\mathbb{C}}

\newcommand{\RCF}{\text{RCF}}

\newcommand{\D}{\mathcal{D}}
\renewcommand{\L}{\mathcal{L}}
\newcommand{\Lring}{\mathcal{L}_\text{\normalfont ring}}

\newcommand{\UCd}{\textrm{\normalfont UC}_{\D}}
\newcommand{\DCF}{\D\textrm{\normalfont -CF}_0}
\newcommand{\Th}{\textrm{\normalfont Th}}

\DeclareMathOperator{\Spec}{Spec}
\DeclareMathOperator{\acl}{acl}
\DeclareMathOperator{\dcl}{dcl}
\DeclareMathOperator{\tp}{tp}
\DeclareMathOperator{\qftp}{qftp}
\DeclareMathOperator{\trdeg}{trdeg}

\renewcommand{\rangle}{\text{\reflectbox{$\langle$}}}

\newcommand{\forkindep}[1][]{%
  \mathrel{
    \mathop{
      \vcenter{
        \hbox{\oalign{\noalign{\kern-.3ex}\hfil$\vert$\hfil\cr
              \noalign{\kern-.7ex}
              $\smile$\cr\noalign{\kern-.3ex}}}
      }
    }\displaylimits_{#1}
  }
}

\newcommand{\algindep}[1][]{%
  \mathrel{
    \mathop{
      \vcenter{
        \hbox{\oalign{\noalign{\kern-.3ex}\hfil$\vert$\rlap{$^\text{alg}$}\hfil\cr
              \noalign{\kern-.7ex}
              $\smile$\cr\noalign{\kern-.3ex}}}
      }
    }\displaylimits_{#1}
  }
}

\newcommand{\Kindep}[1][]{%
  \mathrel{
    \mathop{
      \vcenter{
        \hbox{\oalign{\noalign{\kern-.3ex}\hfil$\vert$\rlap{$^K$}\hfil\cr
              \noalign{\kern-.7ex}
              $\smile$\cr\noalign{\kern-.3ex}}}
      }
    }\displaylimits_{#1}
  }
}

\newcommand{\Lalgindep}[1][]{%
  \mathrel{
    \mathop{
      \vcenter{
        \hbox{\oalign{\noalign{\kern-.3ex}\hfil$\vert$\rlap{$^\L$}\hfil\cr
              \noalign{\kern-.7ex}
              $\smile$\cr\noalign{\kern-.3ex}}}
      }
    }\displaylimits_{#1}
  }
}

\newcommand{\forkindepstar}[1][]{%
  \mathrel{
    \mathop{
      \vcenter{
        \hbox{\oalign{\noalign{\kern-.3ex}\hfil$\vert$\rlap{$^*$}\hfil\cr
              \noalign{\kern-.7ex}
              $\smile$\cr\noalign{\kern-.3ex}}}
      }
    }\displaylimits_{#1}
  }
}

\newcommand{\forkindepplus}[1][]{%
  \mathrel{
    \mathop{
      \vcenter{
        \hbox{\oalign{\noalign{\kern-.3ex}\hfil$\vert$\rlap{$^+$}\hfil\cr
              \noalign{\kern-.7ex}
              $\smile$\cr\noalign{\kern-.3ex}}}
      }
    }\displaylimits_{#1}
  }
}

\title[The uniform companion for large fields with free operators]{The uniform companion for fields with free operators in characteristic zero}
\author{Shezad Mohamed}
\address{Shezad Mohamed, Department of Mathematics, University of Manchester, Oxford Road, Manchester, United Kingdom M13 9PL}
\email{shezad.mohamed@manchester.ac.uk}
\urladdr{https://personalpages.manchester.ac.uk/staff/shezad.mohamed/}

\date{\today}
\thanks{{\em Acknowledgements}: This research was supported by a University of Manchester Research Scholar Award.}

\begin{document}

\begin{abstract}
Generalising the uniform companion for large fields with a single derivation, we construct a theory $\UCd$ of fields of characteristic $0$ with free operators---operators determined by a homomorphism from the field to its tensor product with $\D$, a finite-dimensional $\Q$-algebra---which is the model companion of any theory of a field with free operators whose associated difference field is difference large and model complete. Under the assumption that $\D$ is a local ring, we show that simplicity is transferred from the theory of the underlying field to the theory of the field with operators, and we use this to study the model theory of bounded, PAC fields with free operators.
\end{abstract}

\maketitle

\section{Introduction}\label{sec-intro}

In \cite{tressl_uniform_2005}, Tressl showed that there is an inductive theory UC in the language of differential rings with a single derivation, $\Lring(\delta)$, such that whenever $T$ is a model complete $\Lring$-theory of large fields of characteristic $0$, $T \cup \text{UC}$ is the model companion of the $\Lring(\delta)$-theory of differential fields whose underlying field is a model of $T$. He called UC the uniform companion. This brought the theories $\text{DCF}_0$ and CODF into a common framework, and also showed that the theory of $p$-adically closed fields of fixed $p$-rank with a derivation and the theory of pseudofinite fields with a derivation both admit a model companion.

The results of this paper generalise the above result of Tressl from the case of a single derivation to that of so-called free operators. Fields with free operators, or $\D$-fields as we will now refer to them, were introduced by Moosa and Scanlon in \cite{moosa_scanlon_2010}. In \cite{moosa_scanlon_2013} they showed that $\D$-fields are an appropriate framework for unifying the model theory of various classes of algebraically closed fields with operators: derivations, endomorphisms, and others. We give a brief outline of the set-up; see Section~\ref{sec-preliminaries} for details. Fix a field $k$ of characteristic $0$ and a finite-dimensional $k$-algebra $\D$ with a $k$-algebra homomorphism $\pi \colon \D \to k$. Then a $\D$-field is a field $F$ which is also a $k$-algebra equipped with a $k$-algebra homomorphism $F \to F \otimes_k \D$ which is a section to $\id_F \otimes \pi$. Note that $\delta \colon F \to F$ is a $k$-linear derivation if and only if $x \mapsto x + \delta(x) \varepsilon$ is a $k$-algebra homomorphism $F \to F[\varepsilon]/(\varepsilon^2)$, and $\sigma \colon F \to F$ is a $k$-linear field endomorphism if and only if (somewhat trivially) $x \mapsto (x, \sigma(x))$ is a $k$-algebra homomorphism $F \to F \times F$. In the case $k=\Q$ we recover the usual notions of derivations and endomorphisms. The authors of \cite{moosa_scanlon_2013} then show that the theory of algebraically closed fields of characteristic $0$ with a $\D$-field structure admits a model companion, called $\D\text{-CF}_0$.

Under the mild assumption that $\D$ satisfies Assumption~\ref{res-field-k}, $\D$-fields have a sequence of $t \geq 0$ nontrivial, definable associated endomorphisms. In \cite{Kikyo2002TheSO}, Kikyo and Shelah proved that if $T$ is a model complete theory with the strict order property, $T_\sigma$ has no model companion. An immediate consequence is that, if $t > 0$, $\text{RCF} \cup \text{``$\D$-fields''}$ and $\Th(\Q_p) \cup \text{``$\D$-fields''}$ both have no model companion; we can only hope to find the model companion of a given $\D$-field if its associated difference field has a model companion.

Our aim, therefore, is to find a theory $\UCd$ in the language of $\D$-fields such that whenever $T$ is a model complete theory of difference large fields, $T \cup \UCd$ is the model companion of $T \cup \text{``$\D$-fields''}$. It will then follow that $\D\text{-CF}_0 = \text{ACFA}_{0,t} \cup \UCd$. Here a difference field is difference large if it satisfies a natural modification of the geometric axiom of $\text{ACFA}_{0,t}$; see Definition~\ref{def-difference-large} and Definition~3.2.2 of Cousins's thesis \cite{cousins-thesis} for where this notion first appears. In the case $t=0$ this specialises to the field-theoretic notion of largeness.

Thus, this paper aims to generalise previous results in two senses. In the first sense, we generalise Tressl's uniform companion from the differential case (in a single derivation) to the case of free operators. In the second, the result of Moosa--Scanlon that $\text{ACFA}_{0,t} \cup \text{``$\D$-fields''}$ admits a model companion from the case of $\text{ACFA}_{0,t}$ to that of an arbitrary model complete theory of difference large fields of characteristic $0$. To this end, we prove the following in Section~\ref{sec-uc}.

\begin{theoremnumb}\label{thm-one-intro}
Let $T$ be a model complete theory of difference large fields, and suppose it is the model companion of some $T_0$. Then 
\begin{enumerate}
    \item[\normalfont (i)] $T \cup \UCd$ is the model companion of $T_0 \cup \text{``$\D$-fields''}$;
    \item[\normalfont (ii)] if $T$ is the model completion of $T_0$, then $T \cup \UCd$ is the model completion of $T_0 \cup \text{``$\D$-fields''}$;
    \item[\normalfont (iii)] if $T$ has quantifier elimination, $T \cup \UCd$ has quantifier elimination.
\end{enumerate}
\end{theoremnumb}

When $\D$ is local, there are no associated endomorphisms, and we recover the $\D$-field analogue of Tressl's uniform companion for large differential fields. This yields the uniform companion in the following cases:
\begin{itemize}
    \item several (not necessarily commuting) derivations;
    \item truncated, non-iterative higher derivations;
    \item operators combining these two.
\end{itemize} 

The existence of the uniform companion will follow from two facts:
\begin{enumerate}
    \item[\normalfont (1)] Let $M,N \models \UCd$, and $A$ a common $\D$-subring of them. If $M$ and $N$ have the same existential theory over $A$ as difference fields, then they do as $\D$-fields.
    \item[\normalfont (2)] Every $\D$-field whose associated difference field is difference large can be extended to a model of $\UCd$, and this extension is elementary as an extension of difference fields.
\end{enumerate}
In \cite{tressl_uniform_2005}, Tressl establishes his results for differential fields via these two facts. Our proof of these will be more geometric and takes ideas from \cite{moosa_scanlon_2013}. Indeed, the axiom scheme $\UCd$ is very similar to the geometric axiomatisation of $\D\text{-CF}_0$ in \cite{moosa_scanlon_2013}.

\begin{remark*}
In \cite{tressl_uniform_2005}, Tressl constructs his uniform companion in the setting of several \emph{commuting} derivations $\delta_1, \ldots, \delta_m$. As pointed out above, here we obtain, as an instance of the general result of this paper, the uniform companion in the case of several not necessarily commuting derivations. Thus, in the context of derivations, the results of this paper and those of Tressl differ, agreeing only in the case of a \emph{single} derivation. However, the case of not necessarily commuting derivations does appear in a recent paper of Fornasiero and Terzo \cite{fornasiero-terzo} where they consider generic derivations on algebraically bounded structures---a wider context than the large and model complete fields considered here; see Remark~\ref{slim-alg-bounded-remark} for more details.
\end{remark*}

In Section~\ref{sec-alt-characterisations} we prove alternative characterisations of $\UCd$ in the case that $\D$ is local. One is still geometric in flavour, and one is similar to the notion of differential largeness from \cite{leon_sanchez_differentially_2020}. From these alternative characterisations, we will show that algebraic extensions of large fields which are models of $\UCd$ are again models of $\UCd$ with the unique induced $\D$-structure on the algebraic extension. Hence the algebraic closure of a large field which is a model of $\UCd$ is a model of $\D\text{-CF}_0$ from \cite{moosa_scanlon_2013}. This gives a class of $\D$-fields with minimal $\D$-closures---though it is as yet unknown if $\D$-closures exist in general. Here by a $\D$-closure we mean a prime model extension in the theory $\D\text{-CF}_0$.

The fact that NIP transfers from $T$ to $T \cup \UCd$ is an immediate consequence of the transfer of quantifier elimination. In Section~\ref{sec-transfer} we show the following. 

\begin{theoremnumb}\label{thm-two-intro}
Suppose $\D$ is local and $T$ is a model complete theory of large fields. If $T$ is simple, then $T \cup \UCd$ is simple.
\end{theoremnumb}

This is less immediate and requires the notion of slimness from \cite{junker-koenigsmann} to get a handle on what nonforking independence looks like in the underlying theory of fields. The authors of \cite{junker-koenigsmann} show that model complete, large fields are very slim, and hence that, in such a field, algebraic independence is an independence relation. From the proof of the Kim--Pillay theorem, we then get that if two $\D$-fields are independent in the sense of nonforking, they are algebraically independent as fields. This fact will allow us to amalgamate independent (in the sense of nonforking) $\D$-fields, and thus prove that if the independence theorem holds in $T$, it must also hold in $T \cup \UCd$.

In Section~\ref{sec-pseudo-D-closed} we study the PAC substructures in $\D\text{-CF}_0$ using the definition from \cite{hoffman-dynamics} of being existentially closed in every $\Lring(\partial)$-regular extension (that is, an extension of $\D$-fields $A \leq B$ where $\acl(A) \cap \dcl(B) = \dcl(A)$), and we show that they are characterised as those $\D$-fields which are models of $\UCd$ and PAC as fields. We use this in conjunction with Section~\ref{sec-transfer} to prove simplicity and elimination of imaginaries for the theory of a bounded $\D$-field which is a PAC substructure in $\DCF$, extending the corresponding differential results from Section~\ref{sec-pseudo-D-closed} of \cite{hoffman-leon-sanchez} to the case of $\D$-fields.

In Section~\ref{sec-non-local} we examine what happens without Assumption~\ref{res-field-k}. In this case, $\D$-fields do not necessarily have associated endomorphisms, and it does not make sense to ask whether $T \cup \text{``$\D$-fields''}$ has a model companion when $T$ is a theory of difference fields. So instead we ask whether it has a model companion when $T$ is a theory of fields. We (partially) answer by showing that, when the base field $k$ is finitely generated over $\Q$ and $\D$ is not local, there is some prime $p$ such that $\Th(\Q_p) \cup \text{``$\D$-fields''}$ has no model companion. Thus, when $k$ is finitely generated over $\Q$, the uniform companion \emph{for model complete, large fields} exists if and only if $\D$ is local.

The author would like to thank Omar Le\'{o}n S\'{a}nchez for his very helpful feedback on the results of this paper and his comments on its several drafts.

\textbf{Assumptions. }All rings are commutative and unital. Ring homomorphisms preserve the unit. All fields are of characteristic $0$.

\section{Preliminaries}\label{sec-preliminaries}

In this section we review the basic definitions of $\D$-fields and large fields and the model-theoretic set-up we will be working in. See \cite{moosa_scanlon_2013} and \cite{ozlem_2018} for more details.

\subsection{$\D$-fields}\label{subsec-D-fields}
Fix a base field $k$ of characteristic $0$. Let $\D$ be a finite-dimensional $k$-algebra, and let $\varepsilon_0, \ldots, \varepsilon_l$ be a $k$-basis of $\D$. Let $\pi \colon \D \to k$ be a $k$-algebra homomorphism that sends $\varepsilon_0 \mapsto 1$ and $\varepsilon_i \mapsto 0$ for $i=1, \ldots, l$. If $R$ is a $k$-algebra, $1 \otimes \varepsilon_0, \ldots, 1 \otimes \varepsilon_l$ is an $R$-basis of $R \otimes_k \D$. Write $\pi^R \colon R \otimes_k \D \to R$ for the $k$-algebra homomorphism $\id_R \otimes \pi$.

\begin{definition}
Let $R$ be a $k$-algebra and $\partial_i \colon R \to R$ a sequence of unary functions on $R$ for $i=1, \ldots, l$. We say that $(R,\partial_1, \ldots, \partial_l)$ is a $\D$-ring if the map $\partial \colon R \to R \otimes_k \D$ given by
\[
r \mapsto r \otimes \varepsilon_0 + \partial_1(r) \otimes \varepsilon_1 + \ldots + \partial_l(r) \otimes \varepsilon_l
\]
is a $k$-algebra homomorphism. Equivalently, we will say that $(R, \partial)$ is a $\D$-ring if $\partial \colon R \to R \otimes_k \D$ is a $k$-algebra homomorphism such that $\pi^R \circ \partial = \id_R$.

If $R$ is a $k$-algebra and $S$ is an $R$-algebra given by $a \colon R \to S$, we say that $\partial \colon R \to S \otimes_k \D$ is a $\D$-operator along $a \colon R \to S$ if it is a $k$-algebra homomorphism and $\pi^S \circ \partial = a$. Then $(R, \partial)$ is a $\D$-ring if and only if $\partial$ is a $\D$-operator along $\id_R$.
\end{definition}

The ring structure of $\D$ determines the additive and multiplicative rules of the functions $\partial_i$. Indeed, let $a_{ijk}, b_i \in k$ be the elements defined by $\varepsilon_i \varepsilon_j = \sum_{k=0}^l a_{ijk} \varepsilon_k$ and $1_\D = \sum_{i=0}^l b_i \varepsilon_i$. Then $k$-linearity of $\partial$ corresponds to $k$-linearity of each $\partial_i$. Multiplicativity of $\partial$ corresponds to the following ``product rule'' being satisfied for each $k$: $\partial_k(rs) = \sum_{i,j=0}^l a_{ijk} \partial_i(r) \partial_j(s)$. That $\partial$ preserves the unit corresponds to the equation $\partial_i(1_R) = b_i$.

Note that being a $\D$-ring imposes no additional relations between the functions $\partial_i$. For example, commutativity of the operators is not imposed by being a $\D$-ring (though a particular $\D$-ring may indeed have $\partial_i \partial_j = \partial_j \partial_i$).

We can axiomatise the theory of $\D$-rings in the language
\[
\Lring(\partial) = \{+,-,\cdot, 0, 1, (c_a)_{a \in k}, \partial_1, \ldots, \partial_l \},
\]
where $c_a$ is a constant symbol for the element $a \in k$. 

\begin{example}\label{D-ring-examples}
\begin{enumerate}
    \item Take $\D$ to be the algebra of dual numbers, $k[\varepsilon] / (\varepsilon^2)$, with the standard $k$-algebra structure and basis $1, \varepsilon$. Then $(R, \partial_1)$ is a $\D$-ring precisely when $R$ is a $k$-algebra and $\partial_1$ is a $k$-linear derivation of $R$.
    \item Let $\D = k[\varepsilon_1, \ldots, \varepsilon_l]/(\varepsilon_1, \ldots, \varepsilon_l)^2$ with basis $1, \varepsilon_1, \ldots, \varepsilon_l$. Then $(R, \partial_1, \ldots, \partial_l)$ is a $\D$-ring if $R$ is a $k$-algebra and each $\partial_i$ is a $k$-linear derivation of $R$. As explained before, these derivations will in general be noncommuting.
    \item Take $\D = k^{l+1}$ with the product $k$-algebra structure and the standard basis. Then $(R,\partial_1, \ldots, \partial_l)$ is a $\D$-ring if and only if $R$ is a $k$-algebra and each $\partial_i$ is a $k$-linear endomorphism of $R$. These endomorphisms will in general be noncommuting.
    \item We can combine the above examples. Let $\D = k[\varepsilon] / (\varepsilon^2) \times k$ with basis $(1,0), (\varepsilon,0), (0,1)$. Then a $\D$-ring $(R, \partial_1, \partial_2)$ is a $k$-algebra with a derivation $\partial_1$ and an endomorphism $\partial_2$.
    \item Let $\D = k[\varepsilon] / (\varepsilon^{l+1})$ with basis $1, \varepsilon, \ldots, \varepsilon^l$. Then $\D$-rings are $k$-algebras with non-iterative, truncated higher derivations $(\partial_1, \ldots, \partial_{l})$. That is, they satisfy the following: $\partial_i(xy) = \sum_{r+s=i} \partial_r(x) \partial_s(y)$.
\end{enumerate}
We refer the reader to \cite{moosa_scanlon_2013} for more examples.
\end{example}

Since $\D$ is a finite-dimensional $k$-algebra, it can be written as a finite product of local, finite-dimensional $k$-algebras $\D = \prod_{i=0}^t B_i$. Let $\mathfrak{m}_i$ be the unique maximal ideal of $B_i$. Then the residue field is a finite field extension of $k$: $B_i/\mathfrak{m}_i = k[x]/(P_i)$ for some $k$-irreducible polynomial $P_i$. We define the $k$-algebra homomorphisms $\pi_i \colon \D \to k[x]/(P_i)$ by the compositions $\D \to B_i \to k[x]/(P_i)$, and we let $\pi_i^R = \id_R \otimes \pi_i$ be the $k$-algebra homomorphism $R \otimes_k \D \to R[x]/(P_i)$ for any $k$-algebra $R$. Note that the $k$-algebra homomorphism $\pi \colon \D \to k$ gives a maximal ideal of $\D$ with residue field $k$. So $\pi$ must correspond to one of the $\pi_i$, say $\pi_0$, and $B_0$ has residue field $k$.

\begin{definition}\label{def-ass-hom}
Suppose $\partial \colon R \to S \otimes_k \D$ is a $\D$-operator along $a \colon R \to S$. Composing $\partial$ and the map $\pi_i^S$ gives the following $k$-algebra homomorphism:
\[
\begin{tikzcd}
R \arrow[r, "\partial"] & S \otimes_k \mathcal{D} \arrow[r, "\pi_i^S"] & {S[x]/(P_i)}.
\end{tikzcd}
\]
This is called the $i$th associated homomorphism, $\sigma_i$, of $\partial$.

Now, $\sigma_0 = \pi_0^S \circ \partial = \pi^S \circ \partial = a$ and the associated homomorphism corresponding to $B_0$ is always $a$.

Suppose now that $(R, \partial)$ is a $\D$-ring. If $\alpha \in R$ is a root of $P_i$, we have a map $R[x]/(P_i) \to R$. The composition of $\sigma_i$ with this map gives an endomorphism of $R$, $\sigma_{i,\alpha} \colon R \to R$. This endomorphism is $\alpha$-definable in $\Lring(\partial)$.
\end{definition}

As explained in the introduction, the presence of associated homomorphisms which induce nontrivial endomorphisms (in the sense of the paragraph above) means that theories such as $\RCF \cup \text{``$\D$-fields}$ and $\Th(\Q_p) \cup \text{``$\D$-fields}$ have no model companions. If there are associated homomorphisms which do not induce nontrivial endomorphisms, then we do not know whether these theories have model companions; we defer a more in-depth discussion of this to Section~\ref{sec-non-local}.

In \cite{moosa_scanlon_2013} the authors make the following assumption throughout. 

\begin{assumption}\label{res-field-k}
The residue field of each $B_i$ is $k$.
\end{assumption}

For Section~\ref{sec-uc}, we will do the same. As a result of this assumption, all the associated homomorphisms of a $\D$-ring are in fact endomorphisms, and much is simplified. All of the main motivating examples fit into this situation. Under this assumption, a $\D$-field $(F, \partial)$ has $t$ associated endomorphisms $\sigma_1, \ldots, \sigma_t$. Since the associated endomorphisms are $k$-linear combinations of the operators $\partial_1, \ldots, \partial_l$, we think of the structure $(F, \sigma_1, \ldots, \sigma_t)$---the associated difference field---as a reduct to the language $\Lring(\sigma_1, \ldots, \sigma_t) \subseteq \Lring(\partial)$.

In Sections~4--6, we will make the stronger assumption that $\D$ is local (Assumption~\ref{D-is-local}), and in Section~\ref{sec-non-local} we will discuss what happens in the absence of both assumptions.

\subsection{The prolongation of an affine variety}\label{subsec-prolongation}
For the geometric axioms of the uniform companion, we need the notion of a prolongation of a variety. We will give a similar presentation to \cite{moosa_scanlon_2013}; see \cite{moosa_scanlon_2010} for a more in-depth description of this object as a Weil restriction of a base change, and Remark~2.9 of \cite{moosa_scanlon_ghs} for a construction of the prolongation as an adjunction.

Let $(K,\partial)$ be a $\D$-field and $X$ a variety over $K$. The prolongation of $X$, $\tau X$, is a variety over $K$ with the defining property that for any field extension $L \geq K$, there is an identification $\tau X(L) \leftrightarrow X(L \otimes_k \D)$, where $X$ is viewed as a variety over $K \otimes_k \D$ via the base change coming from the map $\partial \colon K \to K \otimes_k \D$. If $(L, \delta)$ is a $\D$-field extension of $(K, \partial)$, then the above identification induces a map at the level of $L$-points $\nabla \colon X(L) \to \tau X(L)$.

If $X = \Spec K[x]/I$ is affine, where $x$ is a tuple of variables, then $\tau X = \Spec K[x^0, \ldots, x^l] / I'$ where $I'$ is constructed as follows (recall that $\dim_k \D = l+1)$. For each $f \in I$, let $f^\partial \in (K \otimes_k \D)[x]$ be the polynomial obtained by applying $\partial$ to the coefficients of $f$. Now compute
\[
f^\partial \left( \sum_{j=0}^l x^j \varepsilon_j \right) = \sum_{j=0}^l f^{(j)}(x^0, \ldots, x^l) \varepsilon_j .
\]
Then $I'$ is the ideal generated by all the $f^{(j)}$. With respect to these coordinates, the map $\nabla \colon X(L) \to \tau X(L)$ is given by $\nabla(a) = (a, \delta_1(a), \ldots, \delta_l(a))$.

If Assumption~\ref{res-field-k} holds, the $k$-algebra homomorphisms $\pi_i \colon \D \to k$ induce morphisms $\hat{\pi}_i \colon \tau X \to X^{\sigma_i}$, where $X^{\sigma_i}$ is just $X$ base changed via the associated endomorphism $\sigma_i \colon K \to K$. See Section 4.1 of \cite{moosa_scanlon_2010} for a discussion of these morphisms.

If Assumption~\ref{res-field-k} does not hold, but $\alpha \in K$ is a root of $P_i$, then the endomorphism $\sigma_{i, \alpha}$ induces a morphism $\hat{\pi}_{i, \alpha} \colon \tau X \to X^{\sigma_{i,\alpha}}$. This morphism does not follow from Section 4.1 of \cite{moosa_scanlon_2010}---see the discussion after Lemma~7.5 of \cite{moosa_scanlon_2013}. We will not need these morphisms in this paper, but they do appear in the axiomatisation of $\DCF$ when Assumption~\ref{res-field-k} does not hold; see Theorem~7.6 of \cite{moosa_scanlon_2013}.

\begin{example}
\begin{enumerate}
    \item If $\D = k[\varepsilon]/(\varepsilon^2)$, then the prolongation of $X$ is the twisted tangent bundle from \cite{pierce-pillay}. The morphism $\hat{\pi}_0$ is the usual coordinate projection.
    \item If $\D = k \times k$, so that $\D$-rings are precisely rings with a single endomorphism $\sigma$, then the prolongation of $X$ is $X \times X^\sigma$ which appears in the geometric axioms for ACFA in \cite{chatzidakis-hrushovski}. The morphism $\hat{\pi}_0$ is the projection to $X$, and $\hat{\pi}_1$ is the projection to $X^\sigma$.
\end{enumerate}
\end{example}

\begin{remark}
The prolongation of $X$ does not always exist. From Section~\ref{sec-uc}, we will work under the assumption that the residue field of each $B_i$ is $k$. Hence $\Spec B_i \to \Spec k$ is a universal homeomorphism for each $i=0, \ldots, t$, and the prolongation exists for every variety. See the corrigendum to \cite{moosa_scanlon_2010}. In any case, we will only be interested in affine varieties, for which the coordinate construction above suffices.
\end{remark}

\subsection{Large fields and difference large fields}\label{subsec-diff-alg-geo}

\begin{definition}
A field $K$ is large if every $K$-irreducible variety with a smooth $K$-rational point has a Zariski dense set of $K$-rational points. Equivalently, $K$ is large if it is existentially closed in $K((t))$.
\end{definition}

Large fields were first introduced by Pop \cite{pop96}. Most fields considered model-theoretically ``tame'' are large. For example, algebraically closed fields, real closed fields, and fields with a nontrivial henselian valuation are large. On the other hand, number fields and function fields are not large.

Large fields play an important role in Tressl's uniform companion for differential fields in \cite{tressl_uniform_2005}. We will now strengthen this notion to that of \emph{difference largeness}; these difference fields will play an analogous role here.

Recall that $\D$ has a decomposition as $\prod_{i=0}^t B_i$ where each $B_i$ is a local, finite-dimensional $k$-algebra. We impose Assumption~\ref{res-field-k}: the residue field of each $B_i$, which is necessarily a finite field extension of $k$, is $k$ itself. Since our uniform companion will be given ``relative'' to the associated difference field, to simplify notation we will also work with $\mathcal{E}$-operators where
\[
\mathcal{E} = k^{t+1}
\]
so that $\mathcal{E}$-fields are precisely fields with $t$ noncommuting endomorphisms.

Recall also that we have $k$-algebra homomorphisms $\pi_i \colon \D \to k$ given by the composition of the projection to $B_i$ and then the residue map to $k$. Since $k^{t+1}$ is the product, this induces a $k$-algebra homomorphism $\alpha \colon \D \to \mathcal{E}$ which is the product of the maps $\pi_i$. Note also that if $(K, \partial)$ is a $\D$-ring and $(K, \sigma)$ is its associated difference ring thought of as an $\mathcal{E}$-ring (so that $\sigma \colon K \to K^{t+1}$ is given by $r \mapsto (r, \sigma_1(r), \ldots, \sigma_t(r))$), then $\alpha \circ \partial = \sigma$. By Section~4.1 of \cite{moosa_scanlon_2010}, $\alpha$ induces a morphism of varieties $\hat{\alpha} \colon \tau_\D X \to \tau_{\mathcal{E}} X = X \times X^{\sigma_1} \times \ldots \times X^{\sigma_t}$ such that the following diagram commutes:
\[
\begin{tikzcd}
\tau_\D X(K) \arrow[rr, "\hat{\alpha}"] & & \tau_{\mathcal{E}} X (K) \\
& X(K) \arrow[lu, "\nabla_{\D}"] \arrow[ru, "\nabla_{\mathcal{E}}", swap] &
\end{tikzcd}
\]
Note also that $\hat{\alpha}$ is the product of the morphisms $\hat{\pi}_i$.

\begin{lemma}\label{lemma-difference-Z-point-and-ec}
Let $(K, \partial)$ be a $\D$-field, and $(K, \sigma)$ its associated difference field thought of as an $\mathcal{E}$-field. Suppose $X$ and $Y \subseteq \tau_\D X$ are irreducible varieties over $K$. Let $b$ be a $K$-generic point of $Y$. Then the following are equivalent:
\begin{enumerate}[\normalfont (1)]
    \item $Y$ has a Zariski-dense set of $K$-rational points whose projections to $\tau_\mathcal{E} X$ are in the image of $\nabla_\mathcal{E}$;
    \item there is some difference field containing the function field $K(b)$ in which $\hat{\alpha}(b)$ is in the image of $\nabla_\mathcal{E}$ and which is a difference field elementary extension of $(K, \sigma)$.
\end{enumerate}
\end{lemma}
\begin{proof}
(1)$\implies$(2). Working with respect to the coordinates in Section~\ref{subsec-prolongation}, saying that $\hat{\alpha}(b)$ is in the image of $\nabla_\mathcal{E}$ is equivalent to saying that $\sigma_i(\hat{\pi}_0(b)) = \hat{\pi}_i(b)$ for each $i=1, \ldots, t$.

Consider the following set of formulas with parameters in $K$ in the language of difference rings:
\[
p(x) = \qftp(b/K) \cup \{\sigma_i(\hat{\pi}_0(x)) = \hat{\pi}_i(x) \colon i = 1, \ldots, t\}.
\]
Since $b$ is $K$-generic in $Y$ and $Y$ has a Zariski-dense set of $K$-rational points $c$ with $\sigma_i(\hat{\pi}_0(c)) = \hat{\pi}_i(c)$, $p(x)$ is finitely satisiable in $(K, \sigma)$, and hence is a partial type. So there is some difference field $(L, \sigma) \succeq (K, \sigma)$ with a realisation of $p(x)$. This is precisely (2).

(2)$\implies$(1). Clear.
\end{proof}

Recall from the introduction that we cannot hope to uniformly find the model companion for $\D$-fields whose underlying field is large; we need to take into account the associated difference field. The following definition facilitates this.

\begin{definition}\label{def-difference-large}
A difference field $(K, \sigma_1, \ldots, \sigma_t)$ is difference large if it is large and for all irreducible, affine $K$-varieties $X$ and $Y$ such that $Y \subseteq X \times X^{\sigma_1} \times \ldots \times X^{\sigma_t}$, each projection $\hat{\pi}_i \colon Y \to X^{\sigma_i}$ for $i=0, \ldots, t$ is dominant, and $Y$ has a smooth $K$-rational point, we have that $Y$ has a Zariski-dense set of $K$-rational points of the form $(a, \sigma_1(a), \ldots, \sigma_t(a))$ for $a \in X(K)$.
\end{definition}

\begin{remark}\label{rem-difference-large}
\begin{enumerate}
    \item If $t=0$, that is, $\D$ is local, then difference largeness is just largeness. If $t>0$, then the only known examples of difference large fields are models of $\text{ACFA}_{0,t}$; hence we will focus on the local case for consequences and examples in Sections~4--6.
    \item This notion first appeared (for $t=1$) in Cousins's thesis \cite{cousins-thesis}.
\end{enumerate}
\end{remark}

\section{The uniform companion}\label{sec-uc}

Before we discuss the axioms for the uniform companion, we need some results about extending $\D$-ring structures. We carry over the notation from the previous section. In particular, $k$ is a field of characteristic $0$, and $\D$ is a finite-dimensional $k$-algebra. We write the local decomposition of $\D$ as $\prod_{i=0}^t B_i$ and the residue field of $B_i$ as $k[x]/(P_i)$.

\begin{lemma}\label{extending-D-structures}
Suppose $R$ is a $k$-algebra, $S$ is an $R$-algebra given by $a \colon R \to S$, and $\partial \colon R \to S \otimes_k \D$ is a $\D$-operator along $a$ (that is, $\pi^S \circ \partial = a$). Let $\sigma_i \colon R \to S[x]/(P_i)$ be the associated homomorphisms of $\partial$, $\pi_i^S \circ \partial$. Let $\tau_i \colon S \to S[x]/(P_i)$ be $k$-algebra homomorphisms extending $\sigma_i$. If $S$ is $0$-smooth over $R$, there is an extension of $\partial$ to a $\D$-ring structure on $S$ with associated homomorphisms $\tau_i$. If $S$ is $0$-\'{e}tale over $R$, there is a unique such extension.
\end{lemma}
\begin{proof}
Let $(S[x]/(P_i))^{\sigma_i}$ be the $R$-algebra which as a ring is just $S[x]/(P_i)$ but whose $R$-algebra structure is given by $\sigma_i \colon R \to S[x]/(P_i)$.

Let $(S \otimes_k B_i)^\partial$ be the $R$-algebra which as a ring is just $S \otimes_k B_i$ but whose $R$-algebra structure is given by
\[
\begin{tikzcd}
    R \arrow[r, "\partial"] & S \otimes_k \D \arrow[r] & S \otimes_k B_i.
\end{tikzcd}
\]

Similarly, $(S \otimes_k \D)^\partial$ has $R$-algebra structure $\partial \colon R \to S \otimes_k \D$.

Then $\tau_i$ is an $R$-algebra homomorphism $S \to (S[x]/(P_i))^{\sigma_i}$. The quotient map $(S \otimes_k B_i)^\partial \to (S[x]/P_i)^{\sigma_i}$ is a surjective $R$-algebra homomorphism with nilpotent kernel $S \otimes_k \mathfrak{m}_i$. Since $R \to S$ is $0$-smooth ($0$-\'{e}tale), there is a (unique) $R$-algebra homomorphism $S \to (S \otimes_k B_i)^\partial$ whose composition with the quotient map is $\tau_i$. Combining these for each $i$ gives a (unique) $R$-algebra homomorphism $S \to (S \otimes_k \D)^\partial$, that is, a (unique) $\D$-ring structure on $S$ extending $\partial$ with associated homomorphisms $\tau_i$.
\end{proof}

\begin{remark}\label{remark-ext-D-str}
\begin{enumerate}
    \item Separable extensions are $0$-smooth. Separable algebraic extensions are $0$-\'{e}tale. Localisations are $0$-\'{e}tale. See 25.3 and 26.9 of \cite{matsumura_ring} for details.

    \item When $\D$ is local, this lemma appears as Lemma~2.7 in \cite{ozlem_2018}.
\end{enumerate}
\end{remark}

As a consequence of Lemma~\ref{extending-D-structures}, the problem of extending $\D$-field structures to field extensions reduces to that of extending the associated homomorphisms. In general, extending endomorphisms in a given class of fields is not always possible; for instance, in the class of real fields. Thus constructing the model companion of a $\D$-field must be done relative to the associated difference field.

For the rest of this section, we impose Assumption~\ref{res-field-k}.

\begin{definition}
Let $(K,\partial)$ be a $\D$-field. We say that $(K,\partial)$ is a model of $\UCd$ if the following holds:

if $X$ is a $K$-irreducible affine variety and $Y \subseteq \tau_\D X$ is a $K$-irreducible affine variety such that $\hat{\pi}_i (Y)$ is Zariski dense in $X^{\sigma_i}$ for each $i=0, \ldots, t$, and $Y$ has a smooth $K$-rational point, then for every nonempty Zariski open set $U \subseteq Y$ over $K$ there exists some $a \in X(K)$ such that $\nabla(a) \in U(K)$. 
\end{definition}

\begin{remark}
If $(K, \partial) \models \UCd$, then the associated difference field $(K, \sigma)$ is difference large. The proof is similar to that of Proposition~4.12 of \cite{moosa_scanlon_2013}.
\end{remark}

This axiom scheme can be expressed in a first-order fashion in the language $\Lring(\partial)$. This is nowadays a standard argument, but we provide some details; we follow the argument used in \cite{tressl_uniform_2005}. We will make use of Theorem~4.2 from there, which collects results about ideals of polynomials from \cite{van_den_dries_bounds_1984}.

\begin{theorem*}[Theorem~4.2 of \cite{tressl_uniform_2005}]
Let $n,d \in \N$, $x = (x_1, \ldots, x_n)$. Then there are bounds $B = B(n,d)$, $C=C(n,d)$, and $E = E(n,d)$ in $\N$ such that for each field $K$, each ideal $I$ of $K[x]$ generated by polynomials of degree at most $d$, and all $f_1, \ldots, f_p \in K[x]$ of degree at most $d$, the following are true.
\begin{enumerate}
    \item[\normalfont (i)] If $I$ is generated by $f_1, \ldots, f_p$, then for every $g \in I$ of degree at most $d$, there are $c_1, \ldots, c_p \in K[x]$ of degree at most $E$ such that $g = c_1 f_1 + \ldots + c_p f_p$.
    \item[\normalfont (ii)] $I$ is prime if and only if $1 \not \in I$ and for all $f,g \in K[x]$ of degree at most $B$, if $fg \in I$, then $f \in I$ or $g \in I$.
    \item[\normalfont (iii)] For all $m \in \{1, \ldots, n \}$, the ideal $I \cap K[x_1, \ldots, x_m]$ is generated by at most $C$ polynomials of degree at most $C$.
\end{enumerate}
\end{theorem*}

Let $n,d,m \in \N$, $x = (x_1, \ldots, x_n)$, $f_1, \ldots, f_m \in K[x]$, $g_1, \ldots, g_m, h \in K[x^0, \ldots, x^l]$, all of degree at most $d$. The polynomials $f_j$ generate the ideal that defines $X$, $I_X$, the polynomials $g_j$ generate the ideal that defines $Y$, $I_Y$, and the nonvanishing of the polynomial $h$ will define the open subset $U$. Then also the polynomials $f_1^{\sigma_i}, \ldots, f_m^{\sigma_i}$ generate the ideal that defines $X^{\sigma_i}$. The fact that $K$-irreducibility of $X$ and $Y$ can be expressed as first-order $\Lring$-axioms comes from (i) and (ii). That $Y \subseteq \tau X$ can be checked by verifying that the polynomials that define $\tau X$ are elements of $I_Y$; that this is a first-order condition follows from (i). Indeed the polynomials that define $\tau X$ can be computed in terms of the coefficients of the $f_i$ (see the discussion at the end of Section~3 of \cite{moosa_scanlon_2013}). That $\hat{\pi}_0 \colon Y \to X$ is dominant says that $I_Y \cap K[x] = I_X$. By (iii) there is a bound, depending on $n$ and $d$, on the number and degree of the polynomials needed to generate $I_Y \cap K[x]$. That this equality is first-order comes from (i). A similar argument shows that dominance of $\hat{\pi}_i \colon Y \to X^{\sigma_i}$ is a first-order condition. The existence of a smooth $K$-point for $Y$ can be verified by the Jacobian condition on the $g_i$. Note also that the dimension of $Y$ is the Krull dimension of $I_Y$, which is definable in terms of the coefficients of the $g_i$. Finally, that there is a point $\nabla(a)$ in the nonempty Zariski open set $U$ is equivalent to the fact that that the polynomial $h$ either is an element of $I_Y$, which is first-order by (i), or does not vanish at some $\nabla(a)$.

The fact that the theory $\UCd$ is the desired uniform companion will follow from the next two theorems (Theorem~\ref{same-ex-theory} and Theorem~\ref{elext-large}). These form the $\D$-field analogue of Theorem~6.2 from \cite{tressl_uniform_2005}, where the differential counterpart is stated.

\begin{theorem}\label{same-ex-theory}
Let $M,N \models \UCd$ be $\D$-fields, and $A$ a common $\D$-subring. If $M$ and $N$ have the same existential theory over $A$ as difference fields, then they do as $\D$-fields.
\end{theorem}
\begin{proof}
Let $F_1$ and $F_2$ be the quotient fields of $A$ inside $M$ and $N$, respectively. Since the associated endomorphisms of $A$ are injective, they extend uniquely to $F_1$ and $F_2$. By Lemma~\ref{extending-D-structures}, the $\D$-structure on $A$ extends uniquely to $F_1$ and $F_2$, and so they are isomorphic as $\D$-fields. Then we may assume $F = F_1 = F_2$ is contained in both $M$ and $N$. Let $L_1$ be the relative algebraic closure of $F$ in $M$ and similarly for $L_2$ in $N$. Then the associated endomorphisms of $M$ and $N$ restrict to endomorphisms of $L_1$ and $L_2$ respectively, and hence their $\D$-field structures must also. By Remark~\ref{remark-ext-D-str}, $L_1$ and $L_2$ are isomorphic as $\D$-fields; we may then assume $L=L_1=L_2$ is contained in both $M$ and $N$.

Suppose that $M \vDash \exists \bar{x} \phi(\bar{x})$ where $\phi(\bar{x})$ is a quantifier-free $\Lring(\partial)$-formula with parameters in $A$ and $\bar{x} = (x_1, \ldots, x_m)$. As usual, we can assume that $\phi$ is of the form $\bigwedge_{i=1}^n f_i(\bar{x}) = 0$, where the $f_i$ are $\Lring(\partial)$-terms with coefficients in $F$. Let $c_0$ be such that $M \models \phi(c_0)$. Let $\Xi$ be the set of all finite words on $\{\partial_1, \ldots, \partial_l\}$. For each $r$ let $\Xi_r$ be an enumeration of the words of length at most $r$ such that $\Xi_r$ is an initial segment of $\Xi_{r+1}$, and let $n_r = |\Xi_r|$. Let $\nabla_r \colon M \to M^{n_r}$ be given by $b \mapsto (\xi(b) \colon \xi \in \Xi_r)$. Let $r$ be minimal such that $\phi(M) = \{ b \in M^m \colon \nabla_r(b) \in Z \}$ where $Z \subseteq M^{mn_r}$ is a Zariski closed set over $F$. If $r=0$, then $\phi$ is in fact an $\Lring$-formula and since $M$ and $N$ have the same existential theory over $A$ as fields, we have a solution in $N$. So assume $r>0$ and let $c = \nabla_{r-1}(c_0)$.

Let $X = \text{loc}(c / L)$, $Y = \text{loc}(\nabla c / L)$. Note that $Y \subseteq \tau X$ and that $\hat{\pi}_i(Y)$ is Zariski dense in $X^{\sigma_i}$.  Let $g_1, \ldots, g_s$ be polynomials that generate the vanishing ideal of $Y$ over $L$. By the primitive element theorem, let $\alpha \in L$ be such that $F(\alpha)$ is a field of definition for $X$ and $Y$. After clearing denominators, we can rewrite the polynomials $g_i(\bar{u})$ instead as $G_i(v, \bar{u}) \in F[v, \bar{u}]$, where $G_i(\alpha, \bar{u}) = g_i(\bar{u})$. Let $\mu(v) \in F[v]$ be the minimal polynomial of $\alpha$.

We claim that $(\alpha, \nabla c) \in M$ is a smooth solution to the system
\[
\mu(v) = 0, G_1(v, \bar{u})=0, \ldots, G_s(v, \bar{u})=0 . \tag{$\dagger$}
\]

As $\nabla c$ is $L$-generic in $Y$, it must be a smooth solution to the system $g_1(\bar{u})=0, \ldots, g_s(\bar{u})=0$. Let $J(\nabla c)$ be the Jacobian for $g_1, \ldots, g_s$ at $\nabla c$, and let $d$ be its rank. Since $\mu(v)$ contains none of the $\bar{u}$ variables, the Jacobian of the system $(\dagger)$ at $(\alpha, \nabla c)$ is of the form

\[
\mleft(
\begin{array}{c c}
  \frac{d\mu}{dv}(\alpha) & 0 \\
  * & J(\nabla c)
\end{array}
\mright).
\]

Since $\mu$ is the minimal polynomial of $\alpha$, $\frac{d\mu}{dv}(\alpha) \not = 0$, and hence this matrix has rank $d+1$. Note also that the variety defined by $(\dagger)$ is a finite union of conjugates of $Y$, and hence has the same dimension. So $(\alpha, \nabla c)$ is a smooth point of $(\dagger)$.

Consider the quantifier-free $\Lring(\sigma_1, \ldots, \sigma_t)$-type of $(\alpha, \nabla c) \in M$ over $F$. Since $M$ and $N$ have the same existential theory over $F$ as difference fields, this partial type is finitely satisfiable in $N$. We may also assume that $N$ is sufficiently saturated. Then there is a realisation $(\beta, b) \in N$ of this partial type. This induces a difference field $F$-isomorphism $\theta \colon F\langle \alpha \rangle_\sigma \to F \langle \beta \rangle_\sigma$ where $F \langle \alpha \rangle_\sigma$ means the difference field generated by $F$ and $\alpha$ and likewise for $\beta$. We also have that $b$ is a smooth point of $Y^\theta$.

Both $F \langle \alpha \rangle_\sigma$ and $F \langle \beta \rangle_\sigma$ are algebraic extensions of $F$, and hence by Remark~\ref{remark-ext-D-str}, the $\D$-field structure on $F$ extends uniquely to $\D$-field structures on $F\langle \alpha \rangle_\sigma$ and $F\langle \beta \rangle_\sigma$. So $\theta$ is a $\D$-field isomorphism between sub-$\D$-fields of $M$ and $N$.

Since $Y$ is $L$-irreducible, $Y^\theta$ is $L$-irreducible, and since $L$ is relatively algebraically closed in $N$, $Y^\theta$ is $N$-irreducible. Likewise, $X^\theta$ is $N$-irreducible. Proposition~4.8 of \cite{moosa_scanlon_2010} tells us that $\tau X^\theta \simeq \tau X$, and that this isomorphism respects points of the form $\nabla z$. We also have that $Y^\theta \subseteq \tau X^\theta$ and that $\hat{\pi}_i(Y^\theta)$ is Zariski dense in $(X^\theta)^{\sigma_i}$. Since $\theta$ fixes $F$ it also fixes $Z$.

Since $N \vDash \UCd$, there is a point $a \in X^\theta(N)$ with $\nabla a \in Y^\theta(N)$. Let $a_0$ be the first $m$ coordinates of $a$. We claim that $a_0$ is a realisation of $\phi$, which will conclude the proof. As in the proof of Theorem~4.5 of \cite{moosa_scanlon_2013}, we prove first that $\nabla_{r-1}(a_0) = a$. Write $a = (a_\xi \colon \xi \in \Xi_{r-1})$. We prove by induction on the length of $\xi$ that $\xi(a_0) = a_\xi$. For $\xi = \id$ this is clear. Suppose now that $\xi = \partial_i \xi'$. Since $\nabla_{r-1}(c_0) = c$, we have that $\partial_i c_{\xi'} = c_\xi$. This is an algebraic fact about $\nabla c$ over $F$. Since $\nabla a$ satisfies all the algebraic relations $\nabla c$ does over $F$ (since $\theta$ fixes $F$), we also have $\partial_i a_{\xi'} = a_\xi$. By the inductive hypothesis, $\partial_i a_{\xi'} = \partial_i \xi'(a_0) = \xi(a_0)$.

Since $c_0$ realises $\phi$, $\nabla_r(c_0) \in Z$. This is an algebraic fact about $\nabla c = \nabla \nabla_{r-1}(c_0)$ over $F$. Since $\nabla a$ satisfies all the algebraic relations $\nabla c$ does over $F$, we have $\nabla_r(a_0) \in Z$. So $N \models \phi(a_0)$.
\end{proof}

\begin{theorem}\label{elext-large}
Every $\D$-field that is difference large as a difference field has a $\D$-field extension which is a model of $\UCd$ and an elementary extension at the level of difference fields.
\end{theorem}
\begin{proof}
Let $(F, \partial)$ be a $\D$-field that is difference large as a difference field, and let $X$ and $Y$ be $F$-irreducible varieties where $Y \subseteq \tau X$, $\hat{\pi}_i(Y)$ is Zariski dense in $X^{\sigma_i}$, and $Y$ has a smooth $F$-rational point. Let $U$ be a nonempty Zariski-open subset of $Y$ defined over $F$.

Let $b \in Y(L)$ be an $F$-generic point in some field extension $L$ of $F$. Since $\hat{\pi}_i(Y)$ is Zariski dense in $X^{\sigma_i}$, we get that $\hat{\pi}_i(b)$ is $F$-generic in $X^{\sigma_i}$. Let $\hat{\alpha} \colon \tau X \to X \times X^{\sigma_1} \times \ldots \times X^{\sigma_t}$ be the product of the morphisms $\hat{\pi}_i$. Let $Z$ be the Zariski-closure of $\hat{\alpha}(Y)$ in $X \times X^{\sigma_1} \times \ldots \times X^{\sigma_t}$. Then $\hat{\alpha} \colon Y \to Z$ is dominant. Let $V$ be the $F$-open subset of smooth points of $Z$. By dominance, $V$ has a point in the image of $\hat{\alpha}$. Then $\hat{\alpha}^{-1}(V)$ is a nonempty $F$-open set. Since $F$ is large and $Y$ has a smooth $F$-rational point, $\hat{\alpha}^{-1}(V)$ has an $F$-rational point. So $V$ has an $F$-rational point---that is, $Z$ has an $F$-rational smooth point.

Let $W \subseteq Y$ be some open subset of $Y$. Since $Y$ is irreducible, $W$ is dense in $Y$. Then $\hat{\alpha}(W)$ is dense in $Z$. As $Z$ has a smooth $F$-rational point and $F$ is difference large, $Z$ has a Zariski dense set of $F$-rational points of the form $(a, \sigma_i(a), \ldots, \sigma_t(a))$. So $Y$ has a Zariski-dense set of $F$-rational points whose projections have the form $(a, \sigma_i(a), \ldots, \sigma_t(a))$.

So by Lemma~\ref{lemma-difference-Z-point-and-ec} there is some difference field $(K, \sigma)$ containing $F(b)$ which is an elementary extension of $(F, \sigma)$ and in which $\sigma_i(\hat{\pi}_0(b)) = \hat{\pi}_i(b)$. We will define a $\D$-field structure on $K$ whose associated difference field is $(K, \sigma)$.

As mentioned in Section~\ref{subsec-prolongation}, there is an identification $\tau X(K) \leftrightarrow X(K \otimes_k \D)$. Let $b'$ be the tuple from $K \otimes_k \D$ that corresponds to $b \in \tau X(K)$ under this identification. Note that $\pi_i(b') = \sigma_i(\pi_0(b'))$ in $K$ because $\hat{\pi}_i(b) = \sigma_i(\hat{\pi}_0(b))$. Write $a = \hat{\pi}_0(b)$ for the $F$-generic point of $X$.

As in the proof of Theorem~4.5 of \cite{moosa_scanlon_2013}, we can extend $\partial \colon F \to F \otimes_k \D \subseteq K \otimes_k \D$ to a $k$-algebra homomorphism $\partial \colon F[a] \to K \otimes_k \D$ with $\partial(a)=b'$. 

Indeed, since $b \in \tau X(K)$, we have $p^\partial(b') = 0$ for all $p \in I(X/F)$. Since $a$ is $F$-generic in $X$, $I(X/F) = I(a/F)$. As $p^\partial(b') = 0$ for all $p \in I(a/F)$, $\partial$ extends to $F[a] = F[x]/I(a/F)$ by setting $\partial(a) = b'$. Since $\hat{\pi}_0(b) = a$, we have that $\pi_0^K (b') = a$ so $\pi_0^K \circ \partial \colon F[a] \to K$ is inclusion. We also have that $\pi_i \circ \partial (a) = \pi_i(b') = \sigma_i(\pi_0(b')) = \sigma_i(a)$. So $\pi_i \circ \partial = \sigma_i$. Now we can use Lemma~\ref{extending-D-structures} to extend $\partial$ to a $\D$-field structure on $K$ extending the one on $F$ whose associated endomorphisms are precisely the $\sigma_i$. In $(K, \partial)$ we will also have $\nabla(a) = b$. Since $b$ is $F$-generic in $Y$, we must have that $b = \nabla a \in U(K)$.

Then we can iterate this construction transfinitely to get an extension of $F$ that is elementary as an extension of difference fields, which is also a model of $\UCd$.
\end{proof}

\begin{proposition}
The $\Lring(\partial)$-theory $\UCd$ is inductive. If $U$ is an $\Lring(\partial)$-theory of difference large $\D$-fields satisfying the properties in the previous two theorems (Theorems~\ref{same-ex-theory} and \ref{elext-large}), then $U$ contains $\UCd$. If $U$ is in addition inductive, $U = \UCd \cup \text{``difference large fields''}$, where containment and equality here are as deductively closed sets of sentences.
\end{proposition}
\begin{proof}
It is clear that the union of an increasing chain of models of $\UCd$ is also a model of $\UCd$, and likewise for difference large fields. Hence the theories $\UCd$ and ``difference large fields'' are inductive. Using Theorems~\ref{same-ex-theory} and \ref{elext-large}, the rest of the argument is the same as in Proposition~6.3 of \cite{tressl_uniform_2005}.
\end{proof}

Having now established Theorems~\ref{same-ex-theory} and \ref{elext-large}, the following theorem is proved in precisely the same way as its differential counterpart: Theorem~7.1 of \cite{tressl_uniform_2005}.

\begin{theorem}\label{uc-first-theorem}
Let $C$ be a set of new constant symbols and $T$ a model complete theory of difference large fields in the language $\Lring(\sigma)(C)$. Let $T^*$ be a theory in a language $\L^* \supseteq \Lring(\sigma)(C)$ such that $T^* \supseteq T$ and is an extension by definitions of $T$, that is, the only sentences added to $T^*$ are those defining the new symbols in $\L^*$ as $\Lring(\sigma)(C)$-formulas. In addition, let $A$ be an $\mathcal{L}^*(\partial)$-structure such that when $A$ is viewed as an $\Lring(\partial)$-structure it is a $\D$-field.

If $T^* \cup \diag(A \restriction {\mathcal{L}^*})$ is complete, then $T^* \cup \UCd \cup \diag(A)$ is complete. 
\end{theorem}

In the next theorem we argue why Theorem~\ref{uc-first-theorem} justifies calling $\UCd$ the uniform companion for theories of difference large $\D$-fields. This is the precise formulation of Theorem~\ref{thm-one-intro} from the introduction.

\begin{theorem}\label{UC-consequences}
Let $C$ be a set of new constant symbols and $T$ a model complete theory of difference large fields in the language $\Lring(\sigma)(C)$. Let $T^*$ be an $\mathcal{L}^*$-theory which is an extension by definitions of $T$.

Assume $T^*$ is a model companion of an $\mathcal{L}^*$-theory $T_0^*$ which extends the theory of difference fields. Then:
\begin{enumerate}
    \item[\emph{(i)}] $T^* \cup \UCd$ is a model companion of the $\L^*(\partial)$-theory $T_0^* \cup \text{``$\D$-fields''}$;
    \item[\emph{(ii)}] if $T^*$ is a model completion of $T_0^*$, then $T^* \cup \UCd$ is a model completion of $T_0^* \cup \text{``$\D$-fields''}$;
    \item[\emph{(iii)}] if $T^*$ has quantifier elimination, then $T^* \cup \UCd$ has quantifier elimination;
    \item[\emph{(iv)}] if $T$ is complete and $M$ is a $\D$-field which is a model of $T$, then $T^* \cup \UCd \cup \diag(\mathcal{C})$ is complete, where $\mathcal{C}$ is the $\Lring(C)(\partial)$-substructure generated by $\emptyset$ in $M$, that is, the $\D$-subring of $M$ generated by the elements $(c^M)_{c \in C}$.
\end{enumerate}
\end{theorem}
\begin{proof}
First note that $T^* \cup \UCd$ and $T_0^* \cup \text{``$\D$-fields''}$ have the same universal theory (equivalently, that a model of one can be embedded in a model of the other). Let $M \models T^* \cup \UCd$. Since $T^*$ and $T_0^*$ have the same universal theory, there is an $\L^*$-structure $N$ such that $M \restriction_{\L^*} \leq N$. By Lemma~\ref{extending-D-structures}, we can extend the $\D$-structure on $M$ to one on $N$, so that $M \leq N$ as $\L^*(\partial)$-structures. For the other direction, let $M \models T_0^* \cup \text{``$\D$-fields''}$. Then there is some $\L^*$-structure $N \models T^*$ such that $M \restriction_{\L^*} \leq N$. Use Lemma~\ref{extending-D-structures} to extend the $\D$-structure on $M$ to one on $N$, so that $N \models T^* \cup \text{``$\D$-fields''}$ and then use Theorem~\ref{elext-large} to embed this in a model of $T^* \cup \UCd$.

To show (i), it is enough to show that $T^* \cup \UCd$ is model complete, or equivalently, that if $M \models T^* \cup \UCd$, then $T^* \cup \UCd \cup \diag(M)$ is complete. Since $T^*$ is model complete, $T^* \cup \diag(M \restriction {\L^*})$ is complete. Then $T^* \cup \UCd \cup \diag(M)$ is complete by Theorem~\ref{uc-first-theorem}.

For (ii), let $M$ be a model of $T_0^* \cup \text{``$\D$-fields''}$. We need to show that $T^* \cup \UCd \cup \diag(M)$ is complete. But $M \vDash T_0^*$, and so $T^* \cup \diag(M \restriction {\mathcal{L}^*})$ is complete since $T^*$ is a model completion of $T_0^*$. Then apply Theorem~\ref{uc-first-theorem}.

For (iii), let $M \models T^* \cup \UCd$, and let $A \leq M$ be an $\L^*(\partial)$-substructure. We need to show that $T^* \cup \UCd \cup \diag(A)$ is complete. By quantifier elimination for $T^*$, we have that $T^* \cup \diag(A \restriction {\L^*})$ is complete; the result follows by Theorem~\ref{uc-first-theorem}.

For (iv), since $T$ is complete, $T^* \cup \diag(\mathcal{C} \restriction {\L^*}) \subseteq T^*$ is complete. By Theorem~\ref{uc-first-theorem}, $T^* \cup \UCd \cup \diag(\mathcal{C})$ is complete.
\end{proof}

We can now collect the consequences of these theorems, similarly to the differential set-up in Section~8 of \cite{tressl_uniform_2005}. 

\begin{corollary}
\begin{enumerate}[\normalfont (1)]
    \item $\text{\normalfont ACFA}_{0,t} \cup \UCd$ is $\D\text{\normalfont -CF}_0$ from \cite{moosa_scanlon_2013}.
    \item If $\D$ is local, then $\text{\normalfont RCF} \cup \UCd$ is complete and the model companion of the theory of real closed $\D$-fields. It admits quantifier elimination in $\Lring(\leq)(\partial)$.
    \item If $\D$ is local, then $\text{\normalfont $p$CF}_d \cup \UCd$ is complete and the model companion of $p$-adically closed $\D$-fields of fixed rank $d$. It has quantifier elimination in $\Lring(P_n \colon n \in \N)(\partial)$ where $P_n$ is a predicate for the $n$th powers.
    \item Suppose $\D$ is local. Let $\text{\normalfont Psf}_0(C)$ be the $\Lring(C)$-theory of pseudofinite fields of characteristic $0$ with sentences saying that the polynomial $x^n + c_{n,1} x^{n-1} + \ldots + c_{n,n}$ is irreducible for each $n>1$. Recall that $\text{\normalfont Psf}_0(C)$ is model complete. We then get that $\text{\normalfont Psf}_0(C) \cup \UCd$ is the model completion of $\text{\normalfont Psf}_0(C) \cup \text{``$\D$-fields''}$.
\end{enumerate}
\end{corollary}

\section{Alternative characterisations of the uniform companion}\label{sec-alt-characterisations}

In this section we will describe some additional characterisations of $\UCd$ in the case $\D$ is local.  One in particular will use the notion of a D-variety, and will allow us to show that an algebraic extension of a large field which is a model of $\UCd$ is also a model of $\UCd$. In particular, the algebraic closure of such a $\D$-field will be a model of $\D\text{-CF}_0$ from \cite{moosa_scanlon_2013}.

\begin{assumption}\label{D-is-local}
From here until the end of Section~\ref{sec-pseudo-D-closed}, we assume $\D$ is local.
\end{assumption}

\begin{example}
The algebras in (1), (2), and (5) from Example~\ref{D-ring-examples} are local. We can combine these algebras using fibred products and tensor products to form more local examples. See Examples~3.4 and 3.5 of \cite{moosa_scanlon_2013}.
\end{example}

Since $\D$ is local, any $\D$-ring $R$ has only one associated homomorphism: the identity $\id_R$; the associated difference ring is then just the underlying ring. Hence for any affine $K$-variety $X$, there is only one induced morphism $\tau X \to X$, which we call $\hat{\pi}$. With respect to the coordinates described in Section~\ref{subsec-prolongation}, this is just the morphism induced by the inclusion $K[x]/I \to K[x^0, \ldots, x^l]/I'$ where $x \mapsto x^0$.

\begin{definition}
Let $(K, \partial)$ be a $\D$-ring. A D-variety over $(K, \partial)$ is a pair $(V,s)$ where $V$ is a variety over $K$ and $s \colon V \to \tau V$ is an algebraic morphism over $K$ which is a section to the canonical projection $\hat{\pi} \colon \tau V \to V$. We say that $(V,s)$ is $K$-irreducible if $V$ is $K$-irreducible, affine if $V$ is affine, etc.

Given a $\D$-field extension $(L, \delta)$ of $(K, \partial)$, the $(L, \delta)$-rational sharp points of $(V,s)$ are defined as $(V,s)^\sharp (L, \delta) = \{ a \in V(L) \colon \nabla a = s(a) \}$. 
\end{definition}

As before, we will mainly be interested in affine D-varieties. If $V$ is an affine variety, a D-variety structure on $V$ is equivalent to a $\D$-ring structure on its coordinate ring, $K[V]$. A $K$-rational sharp point is equivalent to a $\D$-ring homomorphism $K[V] \to K$. This is the natural $\D$-field analogue of D-varieties as defined for differential rings; see for example \cite{kowalski_pillay_quantifier_2006}.

We now establish some basic results about D-varieties. For the following, if $(R, \partial)$ is a $\D$-ring and $\mathfrak{a}$ is an ideal of $R$, $\mathfrak{a}$ is called a $\D$-ideal if $\partial(\mathfrak{a}) \subseteq \mathfrak{a} \otimes_k \D$, or equivalently, if $\partial_i(\mathfrak{a}) \subseteq \mathfrak{a}$ for each $i=1, \ldots, l$.

\begin{lemma}\label{minimal-primes-are-D-ideals}
Let $(R,\partial)$ be a $\D$-ring, and $\mathfrak{a} \subseteq R$ a radical $\D$-ideal. Then the minimal prime ideals above $\mathfrak{a}$ are $\D$-ideals.
\end{lemma}
\begin{proof}
Let $\mathfrak{p}$ be a minimal prime ideal above $\mathfrak{a}$, and consider the localisation $R_\mathfrak{p}$. Since $\mathfrak{a}$ is radical, so is $\mathfrak{a} R_\mathfrak{p}$ (see Proposition~3.11 of \cite{atiyah_introduction_2018}). Suppose $\mathfrak{q} \subseteq \mathfrak{p} R_\mathfrak{p}$ is a prime ideal of $R_\mathfrak{p}$ that also lies above $\mathfrak{a} R_\mathfrak{p}$. Then by part iv) of the same proposition, we must have $\mathfrak{q} = \mathfrak{p} R_\mathfrak{p}$, and hence $\mathfrak{p} R_\mathfrak{p}$ is a minimal prime above $\mathfrak{a} R_\mathfrak{p}$. It is also the unique maximal ideal of $R_\mathfrak{p}$, and hence is the only prime ideal lying above $\mathfrak{a} R_\mathfrak{p}$. Then $\mathfrak{a} R_\mathfrak{p} = \sqrt{\mathfrak{a} R_\mathfrak{p}} = \mathfrak{p} R_\mathfrak{p}$ (the radical of an ideal is the intersection of the prime ideals lying above it).

By Remark~\ref{remark-ext-D-str} we know that $\partial$ extends uniquely to a $\D$-structure on $R_{\mathfrak{p}}$ with $\partial(\frac{a}{b}) = \partial(a)\partial(b)^{-1}$. Since $\mathfrak{a}$ is a $\D$-ideal it is clear that $\mathfrak{a} R_\mathfrak{p}$ is also a $\D$-ideal.

Then $\mathfrak{p} R_\mathfrak{p}$ is a $\D$-ideal, and hence its contraction to $R$, $\mathfrak{p}$, is also a $\D$-ideal.
\end{proof}

\begin{lemma}\label{open-and-irreducible-D-varieties}
Let $(V,s)$ be an affine D-variety over $(K, \partial)$. Then 
\begin{enumerate}
    \item[\normalfont a)] any nonempty Zariski-open $U \subseteq V$ defined over $K$ is a D-subvariety of $(V,s)$;
    \item[\normalfont b)] any $K$-irreducible component of $V$ is a D-subvariety of $(V,s)$.
\end{enumerate}
\end{lemma}
\begin{proof}
a) Let $K[V]$ be the coordinate ring of $V$. Then $s$ corresponds to $\partial_s \colon K[V] \to K[V] \otimes_k \D$. Let $U$ be a basic open subset of $V$ given by the nonvanishing of some $f$. By Remark~\ref{remark-ext-D-str} we then get that $\partial_s$ extends uniquely to $K[V]_f \to K[V]_f \otimes_k \D$. That is, $s$ restricts to $U \to \tau U$. Now if $U = \bigcup_{i \in I} U_i$ is a union of basic open subsets, $s$ restricts to $U_i \to \tau U_i \subseteq \tau U$, and these restrictions agree on $U_i \cap U_j$ since this is also a basic open. Glueing the morphisms $U_i \to \tau U$ gives a morphism $U \to \tau U$ which is a restriction of $s$.

b) by Lemma~\ref{minimal-primes-are-D-ideals}.
\end{proof}

\begin{theorem}\label{D-large-equivalences}
Suppose $(K,\partial)$ is a $\D$-field and $K$ is large. Then the following are equivalent:
\begin{enumerate}
    \item[\normalfont (1)] $K \vDash \UCd$;
    \item[\normalfont (2)] whenever $(V,s)$ is an affine, $K$-irreducible D-variety, if $V$ has a smooth $K$-rational point, then the set of $K$-rational sharp points of $(V,s)$ is Zariski dense in $V$;
    \item[\normalfont (3)] whenever $(V,s)$ is an affine, $K$-irreducible D-variety, if $V$ has a smooth $K$-rational point, then $(V,s)$ has a $K$-rational sharp point;
    \item[\normalfont (4)] whenever $(V,s)$ is a smooth, affine, $K$-irreducible D-variety, if $V$ has a $K$-rational point, then $(V,s)$ has a $K$-rational sharp point;
    \item[\normalfont (5)] whenever $(L,\delta)$ is a $\D$-field extension of $(K,\partial)$ such that $K$ is existentially closed in $L$ as a field, then $(K,\partial)$ is existentially closed in $(L, \delta)$ as a $\D$-field.
\end{enumerate} 
\end{theorem}
\begin{proof}
(1) $\implies$ (2): Suppose $(K,\partial) \vDash \UCd$ and let $(V,s)$ be a $K$-irreducible D-variety with a smooth $K$-rational point. Let $X = V$ and $Y=s(V)$. Note that $X$ and $Y$ are isomorphic. Then $Y$ has a smooth $K$-rational point, $Y \subseteq \tau X$, and $\hat{\pi} \colon Y \to X$ is an isomorphism. So, since $K \models \UCd$, $Y$ has a Zariski dense set of $K$-rational points of the form $\nabla(a)$ for $a \in X(K)$, and hence for each such $a \in X(K)$, $\nabla(a) = s(a)$.

(2) $\implies$ (3) is clear.

(3) $\implies$ (1): Let $X$, $Y$, $U$ be as in the statement of $\UCd$. Let $b \in L \geq K$ be a $K$-generic point of $Y$, so that $a = \hat{\pi}(b) \in L$ is $K$-generic in $X$ by dominance. Since $b \in \tau X(K(b))$, let $b' \in K(b) \otimes_k \D$ be the point corresponding to $b$ under the correspondence $\tau X (K(b)) \leftrightarrow X(K(b) \otimes_k \D)$. Then $P^\partial(b') = 0$ for all $P \in I(X/K)$, and so $\partial$ extends to a homomorphism $\partial \colon K[a] \to K(b) \otimes_k \D$ with $\partial(a) = b'$. Extend this to a $\D$-ring structure $\partial \colon K(b) \to K(b) \otimes_k \D$ using Lemma~\ref{extending-D-structures}. In this $\D$-ring structure, $\nabla(a) = b$. Now each $\partial_i(b_j) \in K(b)$ so $\partial_i(b_j) = \frac{P_{ij}(b)}{Q_{ij}(b)}$ for some polynomials $P_{ij}, Q_{ij} \in K[x]$. Let $Q \in K[x]$ be the product of all $Q_{ij}$. Note that $\partial$ restricts to $K[b] \to K[b]_{Q(b)} \otimes_k \D$. Again by Lemma~\ref{extending-D-structures}, we must have that $\partial$ extends to $K[b]_{Q(b)} \to K[b]_{Q(b)} \otimes_k \D$. Let $U'$ be the open subset of $Y$ corresponding to $Q(x)$. This extension of $\partial$ gives a D-variety structure $s \colon U' \to \tau U'$.

Since $K$ is large and $V$ has a smooth $K$-point, $U \cap U'$ has a smooth $K$-point. By Lemma~\ref{open-and-irreducible-D-varieties}, $(U \cap U',s|_{U \cap U'})$ is a $K$-irreducible D-variety with a smooth $K$-rational point. By (3) there is $(c, d_1, \ldots, d_{l}) \in (U \cap U')(K)$ with $\nabla(c, d_1, \ldots, d_{l}) = s(c, d_1, \ldots, d_{l})$. Then $c \in X(K)$ with $\nabla(c) = (c, d_1, \ldots, d_{l}) \in U(K)$.

(3) $\implies$ (4) is clear.

(4) $\implies$ (3): Let $(V,s)$ be a D-variety over $K$ with $V$ $K$-irreducible and $a \in V(K)$ a smooth $K$-rational point. Let $W \subseteq V$ be the smooth locus of $V$. Then $W$ is a smooth, $K$-irreducible D-subvariety of $V$. The point $a$ is a $K$-rational point of $W$ and so by (4), $W$ has a $K$-rational sharp point. Then $V$ has a $K$-rational sharp point.

(1) $\implies$ (5): Let $(L, \delta)$ be a $\D$-field extension of $(K, \partial) \models \UCd$ such that $K$ is existentially closed in $L$ as a field. Then there is a field extension $L \leq M$ such that $M$ is an elementary extension of $K$ as a field; note that $M$ is then also a large field. Extend the $\D$-field structure on $L$ to one on $M$, and use Theorem~\ref{elext-large} to find a $\D$-field extension $(N, d) \models \UCd$ such that $K \prec M \prec N$ as fields. This last fact implies that $K$ and $N$ have the same existential theory as fields over $K$. So by Theorem~\ref{same-ex-theory}, they have the same existential theory as $\D$-fields over $(K, \partial)$---recall that since $\D$ is local, the associated difference field is just the underlying field. Then $(K, \partial)$ is existentially closed in $(N, d)$, and hence in $(L, \delta)$.

(5) $\implies$ (1): Assume $(K, \partial)$ has property (5). By Theorem~\ref{elext-large}, there is $(L, \delta) \models \UCd$ extending it such that $K \prec L$ as fields. Then $K$ is existentially closed in $L$ as fields, and so $(K, \partial)$ is existentially closed in $(L, \delta)$ as $\D$-fields by (5). Since $\UCd$ is inductive, we must also have $(K, \partial) \models \UCd$.
\end{proof}

We will now show that algebraic extensions of large models of $\UCd$ are again large and models of $\UCd$. Similar to the differential case (Theorem~5.11 of \cite{leon_sanchez_differentially_2020}), this will rely on the $\D$-Weil descent, established in \cite{mohamed-2022}.

We recall some of the properties of the $\D$-Weil descent. Let $(L,\delta)/(K,\partial)$ be an extension of $\D$-fields where $L/K$ is a finite field extension. Let $(V, s)$ be an affine D-variety over $(L, \delta)$; as mentioned above, this is equivalent to a $\D$-ring structure, $\delta^s$, on the coordinate ring, $L[V]$, extending $\delta$ on $L$. The classical Weil descent of $V$, $V^W$, is a $K$-variety such that there is a natural bijection
\[
V(L) \leftrightarrow V^W(K).
\]

Stated algebraically, this is equivalent to the natural bijection
\[
\text{Hom}_L(L[V], L) \leftrightarrow \text{Hom}_K(K[V^W], K).
\]

In \cite{mohamed-2022} it is shown that there is a unique $\D$-ring structure, $\partial^s$, on $K[V^W]$ extending $\partial$ on $K$ such that the above natural bijection restricts to a natural bijection
\[
\text{Hom}_{(L,\delta)}((L[V], \delta^s), (L,\delta)) \leftrightarrow \text{Hom}_{(K,\partial)}((K[V^W], \partial^s), (K,\partial)).
\]

The $\D$-ring structure $\partial^s$ corresponds to $s^W \colon V^W \to \tau(V^W)$ and makes $(V^W, s^W)$ into a D-variety over $(K, \partial)$. As mentioned above, a $\D$-ring homomorphism $L[V] \to L$ corresponds to an $L$-rational sharp point of $(V,s)$. Geometrically then, we have that the first natural bijection restricts to the natural bijection
\[
(V,s)^\sharp(L, \delta) \leftrightarrow (V^W, s^W)^\sharp(K, \partial).
\]

\begin{theorem}
Let $(L,\delta) / (K,\partial)$ be an algebraic extension of $\D$-fields where $(K,\partial) \models \UCd$ and $K$ is a large field. Then $(L,\delta) \models \UCd$ and $L$ is large.
\end{theorem}
\begin{proof}
Consider first the case when $L/K$ is a finite extension. We verify condition (4) of Theorem~\ref{D-large-equivalences}. Let $(V,s)$ be a smooth, $L$-irreducible D-variety defined over $(L,\delta)$ with an $L$-rational point. Now apply the $\D$-Weil descent to get a D-variety $(V^W, s^W)$ over $(K, \partial)$. Since $V$ is affine and smooth, $V^W$ is affine and smooth (see Proposition~5 of Section~7.6 of \cite{neron_1990}). By the bijection $V(L) \leftrightarrow V^W(K)$, $V^W$ has a $K$-rational point. Let $(U, t)$ be the irreducible component of $(V^W, s^W)$ containing the $K$-rational point. Since $(K, \partial)$ satisfies condition (4), $(U,t)$ has a $K$-rational sharp point, and hence $(V^W, s^W)$ has a $K$-rational sharp point. By the bijection $(V,s)^\sharp(L, \delta) \leftrightarrow (V^W, s^W)^\sharp(K, \partial)$, $(V,s)$ has an $L$-rational sharp point.

If $L/K$ is algebraic, let $F$ be an intermediate extension such that $V$, $s$, and the $L$-rational point are all defined over $F$ and $F/K$ is finite. Then by the above, $(V,s)^\sharp(F, \delta) \not = \emptyset$, and hence $(V,s)^\sharp(L, \delta) \not = \emptyset$.
\end{proof}

\section{Transfer of neo-stability properties}\label{sec-transfer}

We continue to work under Assumption~\ref{D-is-local}. In Section~\ref{sec-uc} we saw that if $T$ is a model complete $\Lring(C)$-theory of large fields and $T^*$ is an expansion by definitions of $T$ with quantifier elimination in a language $\L^*$, then the model companion of $T^* \cup \text{``$\D$-fields''}$, namely $T^* \cup \UCd$, also has quantifier elimination in the language $\L^*(\partial)$. An immediate consequence of this fact is the following.

\begin{corollary}\label{nip-transfer}
Suppose $T$ is the complete, model complete $\Lring(C)$-theory of a large field. If $T$ is NIP, so is (any completion of ) $T \cup \UCd$.
\end{corollary}
\begin{proof}
Let $T^*$ be an expansion by definitions of $T$ with quantifier elimination where the $\L^*$-terms are the same as the $\Lring(C)$-terms (for instance, if $T^*$ is the Morleyisation of $T$). Let $\mathfrak{C}$ be a monster model of $T^* \cup \UCd$.

Suppose $\phi(x,y)$ is an $\L^*(\partial)$-formula with IP: so there are $(a_i)_{i \in \omega}$, $(b_I)_{I \subseteq \omega}$ in $\mathfrak{C}$ with $\mathfrak{C} \models \phi(a_i , b_I) \iff i \in I$. By Theorem~\ref{UC-consequences}, $T^* \cup \UCd$ has quantifier elimination, and we may assume $\phi(x,y)$ is quantifier-free. Now, since the $\L^*$-terms are the same as the $\Lring(C)$-terms, the $\L^*(\partial)$-terms in the variables $x, y$ are then just polynomials in the variables $\{ \nabla_r(x), \nabla_r(y), \nabla_r(c) \colon r \in \N, c \in C \}$. So there are $r \in \N$ and a quantifier-free $\L^*$-formula $\phi^*$ such that $\mathfrak{C} \models \phi(x,y)$ if and only if $\mathfrak{C} \models \phi^*(\nabla_r(x), \nabla_r(y))$. Then
\[
\mathfrak{C} \models \phi^*(\nabla_r(a_i), \nabla_r(b_I)) \iff \mathfrak{C} \models \phi(a_i, b_I) \iff i \in I .
\]
Therefore the tuples $(\nabla_r(a_i))_{i \in \omega}$ and $(\nabla_r(b_I))_{I \subseteq \omega}$ witness that $\phi^*$ has IP.
\end{proof}

\begin{remark}
\begin{enumerate}
    \item This result generalises the fact of Michaux and Rivi{\`e}re that CODF is NIP from Theorem~2.2 of \cite{michaux-riviere-codf}.
    \item In Corollary~4.3 of \cite{GUZY2010570}, Guzy and Point show that NIP is transferred from a topological field (possibly with extra structure) to the model companion of the field with a derivation. The imposition of a topological structure allows them to consider fields with genuine extra structure, as opposed to the definitional expansions considered here.
\end{enumerate}
\end{remark}

A similar argument to Corollary~\ref{nip-transfer} shows that stability transfers via its characterisation of no formula having the order property. But in fact stability yields something stronger.

\begin{lemma}
Suppose $T$ is the complete, model complete $\Lring(C)$-theory of a large field. If $T$ is stable, then $T \cup \UCd = \DCF$.
\end{lemma}
\begin{proof}
A stable, large field of characteristic $0$ is algebraically closed by Theorem D of \cite{etale-open}. The result then follows as $\text{ACF}_0 \cup \UCd = \DCF$.
\end{proof}

We will now prove a similar result to Corollary~\ref{nip-transfer} for the transfer of simplicity. The proof will be via the Kim--Pillay theorem \cite{kim-pillay-simple} and hence we need to understand nonforking independence in $T$. We introduce the notion of very $\L$-slim which is a slight modification of the definition of very slim of Junker--Koenigsmann from \cite{junker-koenigsmann}. First, a restriction on the language we work with and some notation we will need.

\begin{assumption}\label{ass-L}
Let $C$ be a set of constant symbols. For the remainder of this section, we suppose that $\L = \Lring(C)$.
\end{assumption}

Let $K$ be a $\D$-field in the language $\L(\partial)$. For $A \subseteq K$, the field generated by $A$ and $C$ inside $K$ is denoted by $\langle A \rangle_{\L}$ (that is, the quotient field of the $\L$-structure generated by $A$). Similarly, the $\D$-field generated by $A$ and $C$ inside $K$ is denoted by $\langle A \rangle_{\L(\partial)}$. Note then that $\langle XY \rangle_\L = \langle X \rangle_\L \cdot \langle Y \rangle_\L$, where $\cdot$ denotes the compositum in the sense of fields. By the multiplicative rules for $\D$-fields this implies that $\langle X Y \rangle_{\L(\partial)} = \langle X \rangle_{\L(\partial)} \cdot \langle Y \rangle_{\L(\partial)}$.

\begin{definition}
Let $K$ be a field in the language $\L$. We say that $K$ is $\L$-slim if for every $\L$-substructure $F$, we have $\acl_{\L}^K(F) = F^\text{alg}$. Equivalently, if for every subset $A$, we have $\acl_{\L}^K(A) = \langle A \rangle_\L^\text{alg}$. By $F^\text{alg}$ we always mean the relative, field-theoretic algebraic closure of $F$ in $K$.

We say that $K$ is very $\L$-slim if every $\L$-structure elementarily equivalent to $K$ is $\L$-slim.
\end{definition}

\begin{remark}\label{very-L-slim-remark}
\begin{enumerate}
    \item We recover the definition of (very) slim from \cite{junker-koenigsmann} by setting $C$ to be empty, and only considering fields in the language of rings. In other words, a field with no extra structure is (very) slim exactly when it is (very) $\Lring$-slim.
    \item Let $\mathcal{C}$ be the field generated by the constants $C$ inside $K$. Then by Lemma~2.12 of \cite{johnson-ye-geo}, $K$ is very $\L$-slim if and only if it is algebraically bounded over $\mathcal{C}$.
\end{enumerate}
\end{remark}

As mentioned in \cite{junker-koenigsmann} for slim fields, to check whether $K$ is very $\L$-slim it is enough to check whether a sufficiently saturated model of its theory is $\L$-slim.

We take the following definition from Kim--Pillay \cite{kim-pillay-simple}, though phrased in the terminology of Adler \cite{adler_2009}.

\begin{definition}
Let $\mathfrak{C}$ be a saturated and strongly homogeneous structure. A relation $\forkindepstar$ on triples of small subsets of $\mathfrak{C}$ is called an independence relation if it is invariant under automorphisms and satisfies: 
\begin{enumerate}
    \item \emph{normality:} $X \forkindepstar[A] B \implies X \forkindepstar[A] AB$;
    \item \emph{monotonicity:} $X \forkindepstar[A] B \implies X \forkindepstar[A] B'$ for $B' \subseteq B$;
    \item \emph{base monotonicity:} $X \forkindepstar[A] D \implies X \forkindepstar[B] D$ for $A \subseteq B \subseteq D$;
    \item \emph{transitivity:} $X \forkindepstar[A] B \text{ and } X \forkindepstar[B] D \implies X \forkindepstar[A] D$ for $A \subseteq B \subseteq D$;
    \item \emph{symmetry:} $X \forkindepstar[A] B \iff B \forkindepstar[A] X$;
    \item \emph{full existence:} for any $X, A, B$ there is $X' \equiv_A X$ such that $X' \forkindepstar[A] B$ (here and throughout $X' \equiv_A X$ means that $X'$ and $X$ have the same type over $A$);
    \item \emph{finite character:} if $X_0 \forkindepstar[A] B$ for all finite $X_0 \subseteq X$, then $X \forkindepstar[A] B$;
    \item \emph{local character:} there is a cardinal $\kappa$ such that for all $X$ and $A$, there is $A_0 \subseteq A$ with $|A_0| < \kappa$ such that $X \forkindepstar[A_0] A$.
\end{enumerate}

An independence relation $\forkindepstar$ satisfies the independence theorem over $M$ if the following holds:
\begin{enumerate}
    \item[(9)] \emph{independence theorem over $M$:} if $A_1 \forkindepstar[M] A_2$, $a_1 \forkindepstar[M] A_1$, $a_2 \forkindepstar[M] A_2$, and $a_1 \equiv_M a_2$, then there is $a \models \tp(a_1/M A_1) \cup \tp(a_2/M A_2)$ with $a \forkindepstar[M] A_1 A_2$.
\end{enumerate}
\end{definition}

In simple theories, nonforking independence is an independence relation that satisfies the independence theorem over models. As shown in \cite{junker-koenigsmann}, algebraic independence in very slim fields is an independence relation---in general, algebraic independence does not satisfy full existence (called existence in \cite{junker-koenigsmann}).

We recall the definition of algebraic independence in fields. If $\mathbb{U}$ is a big field and $A$, $B$, $D$ are subfields with $D \leq A,B$, we say $A$ and $B$ are algebraically independent over $D$ if every finite tuple from $A$ which is algebraically independent over $D$ is also algebraically independent over $B$. We write $A \algindep[D] B$ if this holds. 

If $A$ and $B$ are linearly disjoint over $D$, then they are algebraically independent. The converse holds if at least one of the extensions $A/D$ or $B/D$ is regular, or equivalently, since we are in characteristic $0$, relatively algebraically closed.

\begin{definition}\label{def-L-alg-independence}
For an $\L$-structure $K$, define the following relation on triples of subsets of $K$:
\[
A \Lalgindep[D] B \iff \langle A \rangle_\L \algindep[\langle D \rangle_\L] \langle B \rangle_\L .
\]
We say that $A$ and $B$ are $\L$-algebraically independent over $D$.
\end{definition}

Theorem~2.1 of \cite{junker-koenigsmann} says that a field is very slim if and only if algebraic independence is an independence relation. The following result is the very $\L$-slim analogue.

\begin{lemma}
$K$ is very $\L$-slim if and only if $\Lalgindep$ is an independence relation.
\end{lemma}
\begin{proof}
Invariance, monotonicity, base monotonicity, symmetry, and transitivity are clear. The fact that $\langle XY \rangle_\L = \langle X \rangle_\L \cdot \langle Y \rangle_\L$ implies normality. Indeed, $A \Lalgindep[D] B \implies \langle A \rangle_\L \algindep[\langle D \rangle_\L] \langle B \rangle_\L$. Then by normality for $\algindep[\ \ ]$ we have $\langle A \rangle_\L \algindep[\langle D \rangle_\L] \langle B \rangle_\L \langle D \rangle_\L$, so $\langle A \rangle_\L \algindep[\langle D \rangle_\L] \langle B D \rangle_\L$ and $A \Lalgindep[D] BD$. Finite and local character also follow from the corresponding property for algebraic independence.

Now the same proof as in Theorem~2.1 of \cite{junker-koenigsmann} shows that $K$ is very $\L$-slim if and only if algebraic independence satisfies full existence when the base is an $\L$-structure.
\end{proof}

\begin{remark}
Remark 1.20 of \cite{adler2005} says that nondividing independence (which is not in general an independence relation) implies any independence relation. Indeed this fact is implicit in the proof of the Kim--Pillay theorem, see Theorem~4.2 Claim~I of \cite{kim-pillay-simple}. Hence if $T$ is the theory of a simple, very $\L$-slim field, nonforking independence implies $\L$-algebraic independence.
\end{remark}

The following result will allow us to use the notion of very $\L$-slimness in the context of Section~\ref{sec-uc}.

\begin{lemma}
Let $K$ be an $\L$-structure which is a field of characteristic $0$. Suppose $K$ is large and model complete (that is, its $\L$-theory is model complete). Then $K$ is very $\L$-slim.
\end{lemma}
\begin{proof}
The same proof as in Theorem~5.4 of \cite{junker-koenigsmann} works here. In part 2 of that proof, when they take a subfield $k$, we instead take an $\L$-substructure $k$, that is, a subfield containing the constants $C$. Model completeness then implies that $\phi$ is an existential $\Lring(C)$-formula with parameters from $k$. But this is the same as an existential $\Lring$-formula with parameters from $k$ since $k$ contains $C$. In part 5 of the proof, we do not need perfectness since we are in characteristic $0$. The rest of the proof is the same.
\end{proof}

\begin{remark}\label{slim-alg-bounded-remark}
As noted in Remark~\ref{very-L-slim-remark}(2), $K$ is very $\L$-slim if and only if it is algebraically bounded over the field generated by the constants $C$. Thus the results of this paper, specialised to the case of noncommuting derivations, appear as an instance of the work of Fornasiero--Terzo \cite{fornasiero-terzo}; there arbitrary expansions of fields are considered, here only expansions by constants.
\end{remark}

We now fix some notation that will be in place for the rest of this section. Let $T$ be the model complete $\L$-theory of a large field of characteristic $0$, $T_\D$ the $\L(\partial)$-theory $T \cup \text{``$\D$-fields''}$, and $T^+ = T \cup \UCd$ the model companion of $T_\D$, which we know exists by Section~\ref{sec-uc}. By $\acl$ we mean the model theoretic algebraic closure in the sense of $T^+$, and by $\acl_T$ we mean in the sense of $T$.

We will use the following result about amalgamating $\D$-field structures.

\begin{fact}\label{amalgamate-D-structures}
Suppose $(K, \partial)$ and $(L, \gamma)$ are two $\D$-fields containing a common $\D$-subfield $(F, \partial)$ with $K$ and $L$ linearly disjoint over $F$ inside some common field extension (which is not necessarily a $\D$-field). Then there is a unique $\D$-structure on the compositum $KL$ extending $\partial$ and $\gamma$.
\end{fact}
\begin{proof}
This is just Lemma~5.1 of \cite{moosa_scanlon_2013}. We drop the inversiveness condition since $\D$ is local so there are no (nontrivial) associated endomorphisms.
\end{proof}

\begin{lemma}\label{structure-of-acl}
Let $(K, \partial) \models T^+$. For $A \subseteq K$, $\acl(A) = \acl_T \left( \langle A \rangle_{\L(\partial)} \right)$.
\end{lemma}
\begin{proof}
Let $F = \acl_T \left( \langle A \rangle_{\L(\partial)} \right)$. Clearly $F \subseteq \acl(A)$. For the other containment, suppose $d \not \in F$. We will show that $\tp(d/F)$ is not algebraic by finding infinitely many realisations.

The $\L$-structure $K$ is very $\L$-slim so $\L$-algebraic independence in $K$ is an independence relation. By full existence there is an $\L$-structure $K'$ which is an $\L$-elementary extension of $K$ containing $L$ such that $L \equiv_F^{\L} K$ (recall that this means that $L$ and $K$ have the same $\L$-type over $F$) with $L$ and $K$ $\L$-algebraically independent over $F$ inside $K'$. Now $F$ is relatively algebraically closed in $K$, so $L$ and $K$ are linearly disjoint over $F$ inside $K'$. Note also that $F$ is a $\L(\partial)$-structure and $L$ has a $\D$-field structure coming from the partial $\L$-elementary map $K \to L$. By Fact~\ref{amalgamate-D-structures} we can amalgamate their $\D$-field structures to the compositum $LK$ and finally we extend this $\D$-field structure to $K'$ using Remark~\ref{remark-ext-D-str}. So $K' \models T_\D$. Since $T^+$ is the model companion of $T_\D$, embed $K'$ inside some $K'' \models T^+$. Let $\alpha \colon K \to L$ be the $\L(\partial)$-isomorphism fixing $F$. By model completeness of $T^+$, this is an $\L(\partial)$-elementary map. Then $\alpha(d) \models \tp(d/F)$ and $\alpha(d) \not = d$ as otherwise, $d \algindep[F] d$ which would imply $d \in F^\text{alg} = F$.

Now we can iterate to find the infinitely many realisations of $\tp(d/F)$.
\end{proof}

We are now in a position to prove the transfer of simplicity. The proof idea is standard.

\begin{theorem}\label{simplicity-transfers}
If $T$ is simple, then (any completion of) $T^+$ is simple.
\end{theorem}
\begin{proof}
Let $\forkindep$ be nonforking independence for $T$. For a monster model $\mathfrak{C}$ of $T^+$ and small subsets $A,B,D$ of $\mathfrak{C}$, we define $\forkindepplus$ as follows:
\[
A \forkindepplus[D] B \iff \acl(A) \forkindep[\acl(D)] \acl(B) .
\]

We will show that $\forkindepplus$ satisfies the conditions of the Kim-Pillay theorem, and hence that $T^*$ is simple.

\emph{Invariance, monotonicity, transitivity, symmetry, finite character, local character.} These follow from the corresponding property of $\forkindep$. For \emph{normality}, the comment after Assumption~\ref{ass-L} gives $\langle BD \rangle_{\L(\partial)} = \langle B \rangle_{\L(\partial)} \cdot \langle D \rangle_{\L(\partial)}$. Then Lemma~\ref{structure-of-acl} gives $\acl(BD) \subseteq \acl_T(\acl(B), \acl(D))$.

\emph{Full existence.} Let $a, A, B$ be given inside some $M \models T^+$. We will find $a'$ in some elementary extension of $M$ with $a' \equiv_A a$ and $a' \forkindepplus[A] M$, which implies $a' \forkindepplus[A] B$ by monotonicity. We may also assume that $A = \acl(A)$, and hence that $A$ is a $\D$-field and is algebraically closed in the sense of $T$. Note that $A$ is field-theoretically, relatively algebraically closed in $M$. Since $T$ is simple, use full existence for $\forkindep$ to find a saturated and homogeneous $\L$-elementary extension $N$ of $M$ such that $M' \forkindep[A] M$ with $M' \subseteq N$ and $M' \equiv_A^{\L} M$. Use the $\L$-isomorphism over $A$ to put a $\D$-field structure on $M'$. Now $M$ and $M'$ are algebraically independent over $A$, and hence linearly disjoint over $A$ since $A \subseteq M$ is a regular extension. They then amalgamate to a $\D$-field structure on the compositum $MM'$, which we may then extend to a $\D$-field on $N$. So $N \models T_\D$ and extends to $N' \models T^+$. Let $a' \in M'$ correspond to $a \in M$. By model completeness of $T^+$, the $\L(\partial)$-isomorphism $M \to M'$ is an $\L(\partial)$-elementary map. Then $a' \equiv_A a$ in the sense of $\L(\partial)$. Since $M' \models T^+$ and $T^+$ is model complete, it is $\acl$-closed, and so $\acl(a'A) \subseteq M'$. By monotonicity, $\acl(a'A) \forkindep[\acl(A)] \acl(M)$ and so $a' \forkindepplus[A] M$.

\emph{Independence theorem.} Let $M \prec N \models T^+$, $A_1 \forkindepplus[M] A_2$, $a_1 \forkindepplus[M] A_1$, $a_2 \forkindepplus[M] A_2$, and $\tp(a_1/M) = \tp(a_2/M)$. We will show that in some elementary extension of $N$, there is $a \forkindepplus[M] A_1 A_2$ realising $\tp(a_1/A_1) \cup \tp(a_2/A_2)$. By Löwenheim--Skolem and full existence for $\forkindepplus$, we may assume that $A_1, A_2, a_1, a_2$ are all models of $T^+$ containing $M$ and contained inside $N$.

By the independence theorem for $T$, there are an $\L$-elementary, saturated, homogeneous extension $N'$ of $N$ and an element $a \in N'$ with $a \forkindep[M] A_1 A_2$ and $a \models \tp_\L(a_1/A_1) \cup \tp_\L(a_2/A_2)$. In fact, by full existence for $T$, we can ensure that $a \forkindep[M] N$.

Now for $i=1,2$, let $N_i'$ be the copy of $N$ coming from the $\L$-elementary map $A_i a_i \mapsto A_i a$. By full existence for $T$, let $N_i \equiv_{A_i a}^{\L} N_i'$ with $N_1 \forkindep[A_1 a] N$ and $N_2 \forkindep[A_2 a] NN_1$. From $a \forkindep[M] N$ we get $a \forkindep[A_1] N$ and $a \forkindep[A_2] N$. Then $N \forkindep[A_1] N_1$ and $N \forkindep[A_2] N_2$ by transitivity. From $a \forkindep[M] N$ we get $a \forkindep[A_1] A_2$, and so $A_1 a \forkindep[A_1] A_2$. Along with $A_1 \forkindep[M] A_2$, transitivity gives $A_1 a \forkindep[M] A_2$, so that $A_1 a \forkindep[a] A_2$ by base monotonicity. This implies $A_1 \forkindep[a] A_2$ and $N_1 \forkindep[a] A_2$. This last part implies $N_1 \forkindep[a] A_2 a$ and along with $N_2 \forkindep[A_2 a] N N_1$ implies $N_1 \forkindep[a] N_2$. Also, $N \forkindep[A_1 A_2] N_1$ by base monotonicity since $A_1 A_2 \subseteq N$. From $N N_1 \forkindep[A_2 a] N_2$, we get $N \forkindep[A_2 a N_1] N_2$, and hence $N \forkindep[A_2 N_1] N_2$ since $a \in N_1$. Combining this with $N \forkindep[A_1 A_2] N_1$ gives $N \forkindep[A_1 A_2] N_1 N_2$.

Now define $\D$-field structures $\partial_1$ on $N_1$ and $\partial_2$ on $N_2$ such that $(N_i, A_i, a, \partial_i)$ is $\L(\partial)$-isomorphic to $(N, A_i, a_i, \partial)$. So $N_i \models T^+$. Now since $N_1 \forkindep[a] N_2$, $N_1$ and $N_2$ are algebraically independent over $a$, and hence linearly disjoint over $a$ (since both are regular extensions of $a$) and so their $\D$-field structures can be amalgamated to the compositum $N_1 N_2$. Extend the $\D$-field structure on $N_1 N_2$ to $\acl_T(N_1 N_2) = (N_1 N_2)^\text{alg}$. From $N \forkindep[A_1 A_2] N_1 N_2$ we get $N \forkindep[\acl_T(A_1 A_2)] \acl_T(N_1 N_2)$ and hence we may amalgamate the $\D$-field structures on $\acl_T(N_1 N_2)$ and $N$ to the compositum $\acl_T(N_1 N_2) N$ and extend this to $N'$. So $N' \models T_\D$.

Extend $N'$ to $N'' \models T^+$. By model completeness of $T^+$, the $\L(\partial)$-isomorphisms $(N_i, A_i, a, \partial_i) \simeq (N, A_i, a_i, \partial)$ are $\L(\partial)$-elementary maps, and hence $a \models \tp(a_1/A_1) \cup \tp(a_2/A_2)$. From $a \forkindep[M] N$ and monotonicity, we get $a \forkindep[M] \acl(A_1 A_2)$, and so $a \forkindepplus[M] A_1 A_2$.
\end{proof}

\begin{remark}
In \cite{chatzidakis_1999}, Chatzidakis proves that for any field $F^*$, if $F \prec F^*$, $A$ and $B$ are $\acl$-closed, and $A \forkindep[F] B$, then $A$ and $B$ are linearly disjoint over $F$. This fact allows us to amalgamate $\D$-fields which are independent over models. However, in the above proofs (and in Theorem~\ref{thm-simplicity-and-ei} below), we must amalgamate $\D$-fields which are independent over arbitrary $\acl$-closed sets. Hence the notion of $\L$-slimness to facilitate this.
\end{remark}

\section{Pseudo $\D$-closed fields}\label{sec-pseudo-D-closed}

We now apply some of the results obtained in Sections~4 and 5 to the study of PAC substructures in the theory $\D\text{-CF}_0$. We continue to assume that $\D$ is local (Assumption~\ref{D-is-local}) and that we are working in a language $\L$ of the form $\Lring(C)$ (Assumption~\ref{ass-L}). Recall that a field $K$ is called pseudo algebraically closed (PAC) if every absolutely irreducible variety over $K$ has a $K$-rational point. PAC fields are large---a $K$-irreducible variety with a smooth $K$-rational point is absolutely irreducible---and a field is PAC if and only if it is existentially closed in every regular extension.

In \cite{chatzidakis_pillay}, Chatzidakis and Pillay show that the $\L$-theory of a bounded, PAC field is simple, and that if it, in addition, has finite degree of imperfection, then it eliminates imaginaries (after naming constants). Recall that a field is bounded if it has only finitely many separable algebraic extensions of each degree. Hoffman and Le\'{o}n S\'{a}nchez in \cite{hoffman-leon-sanchez} then prove the analogous results for bounded, pseudo differentially closed fields of characteristic $0$. Their result gives an example of a differential field whose theory is simple and unstable. In this section we will prove analogous results in the case of $\D$-fields.

PAC substructures of a given theory have been defined as generalisations of PAC fields in various ways. We use the definition presented in \cite{hoffman-dynamics}.

\begin{definition}
Let $T$ be an arbitrary complete $L$-theory, and $\mathfrak{C}$ a monster model. An extension of $L$-substructures $A \leq B$ of $\mathfrak{C}$ is called $L$-regular if $\dcl^\text{eq}(B) \cap \acl^\text{eq}(A) = \dcl^\text{eq}(A)$. An $L$-substructure $A$ of $\mathfrak{C}$ is called a PAC substructure if $A$ is existentially closed in every $L$-regular extension.
\end{definition}

Consider now the $\Lring(\partial)$-theory $\DCF$. This theory eliminates imaginaries (see Theorem~5.12 of \cite{moosa_scanlon_2013}), $\dcl(A)$ is the $\D$-field generated by $A$, and $\acl(A)$ is the (full) field-theoretic algebraic closure of the $\D$-field generated by $A$ (Proposition~5.5 of \cite{moosa_scanlon_2013}). Then an extension of $\D$-fields is $\Lring(\partial)$-regular exactly when the field extension is field-theoretically, relatively algebraically closed (and so regular in the field sense since we are in characteristic $0$).

We now prove three conditions equivalent to being a PAC substructure in $\DCF$.

\begin{theorem}\label{thm-pac-equivalents}
Let $(K, \partial)$ be a $\D$-field. The following are equivalent:
\begin{enumerate}
    \item[\normalfont (1)] $(K, \partial)$ is a PAC substructure in the theory $\DCF$;
    \item[\normalfont (2)] $K$ is a PAC field and $(K, \partial) \models \UCd$;
    \item[\normalfont (3)] if $(V,s)$ is a D-variety over $K$ and $V$ is absolutely irreducible, then $(V,s)$ has a $K$-rational sharp point;
    \item[\normalfont (4)] $(K, \partial)$ is existentially closed in every $\D$-field extension $(L, \delta)$ which is R-regular, that is, where $\tp^{\DCF}(a/K)$ is stationary for every finite tuple $a \in L$.
\end{enumerate}
\end{theorem}
\begin{proof}
(1)$\implies$(2). Let $L$ be any regular field extension of $K$, and let $\delta$ be any $\D$-structure on $L$ extending $\partial$. Then $(K, \partial)$ is existentially closed in $(L, \delta)$ as $\D$-fields, and hence $K$ is existentially closed in $L$ as fields. So $K$ is PAC. Now since $K$ is large, there is a $\D$-field extension $(F, \gamma) \models \UCd$ of $(K, \partial)$ such that $K$ is elementary in $F$ as fields. In particular, $K \subseteq F$ is regular. By (1), $(K, \partial)$ is existentially closed in $(F, \gamma)$. Since $\UCd$ is inductive, $(K, \partial) \models \UCd$.

(2)$\implies$(1). Let $(L, \delta)$ be an $\Lring(\partial)$-regular $\D$-field extension of $(K, \partial)$, so that $L/K$ is a regular field extension. Since $K$ is PAC, $K$ is existentially closed in $L$ as fields. By characterisation (5) of Theorem~\ref{D-large-equivalences}, $(K, \partial)$ is existentially closed in $(L, \delta)$ as $\D$-fields.

(2)$\implies$(3). If $V$ is absolutely irreducible, then $K$ is regular in its function field $K(V)$. $V$ has a smooth $K(V)$-rational point, and since $K$ is existentially closed in $K(V)$, $V$ has a smooth $K$-rational point. By characterisation (3) of $\UCd$ in Theorem~\ref{D-large-equivalences}, $(V,s)$ has a $K$-rational sharp point.

(3)$\implies$(2). Let $V$ be an absolutely irreducible variety defined over $K$. Extend the $\D$-field structure on $K$ to one on $K(V)$. As in the proof of Theorem~\ref{D-large-equivalences} (3) $\implies$ (1), there is an affine, open subset $U \subseteq V$ defined over $K$ such that this $\D$-ring structure restricts to one on $K[U]$. This gives a D-variety structure $s$ on $U$, making $(U,s)$ an absolutely irreducible D-variety defined over $(K, \partial)$. By (3), $(U,s)$ has a $K$-rational sharp point, and hence $V$ has a $K$-rational point. So $K$ is a PAC field. We again use characterisation (3) of Theorem~\ref{D-large-equivalences} and the fact that a $K$-irreducible variety with a smooth $K$-rational point is absolutely irreducible to get that $(K, \partial) \models \UCd$.

(1)$\iff$(4) is the content of Lemma~3.36 in \cite{hoffman-dynamics}; R-regular extensions are the same as $\Lring(\partial)$-regular extensions since $\DCF$ is stable and eliminates imaginaries.
\end{proof}

We say that a $\D$-field is pseudo $\D$-closed if any of the equivalent conditions of Theorem~\ref{thm-pac-equivalents} hold.

\begin{remark}
Apart from condition (3), this is just the $\D$-field analogue of Theorem~5.16 from \cite{leon_sanchez_differentially_2020}. There the authors need to consider differential varieties as they work with several commuting derivations. In a single derivation, it is enough to consider D-varieties; see Proposition~5.6 of \cite{pillay-polkowska-pac} for instance.
\end{remark}

Theorem~5.2 of \cite{hoffman-leon-sanchez} states that the theory of a bounded, pseudo differentially closed field (that is, a PAC substructure of $\text{DCF}_{0,m}$) is simple and eliminates imaginaries. We will now prove the $\D$-field analogue. Let $(K, \partial)$ be a bounded, pseudo $\D$-closed field. For each $n>1$, let $N(n)$ be the degree over $K$ of the Galois extension composite of all Galois extensions of $K$ of degree $n$. Let $C = (c_{n,i})_{n>1, 0 \leq i < N(n)}$ be the set of constant symbols in our language $\L = \Lring(C)$, and consider the set of $\L$-sentences $\Sigma_C = \{ \sigma_n \colon n > 1 \}$ where $\sigma_n$ says that the polynomial $x^{N(n)} + c_{n, N(n)-1} x^{N(n)-1} + \ldots + c_{n,0}$ is irreducible and the extension this polynomial defines is Galois and contains all Galois extensions of $K$ of degree $n$. This is the same set-up used by Chatzidakis and Pillay in Section~4 of \cite{chatzidakis_pillay} in their treatment of bounded, PAC fields. Let $T^+ = \Th(K, \partial) \cup \Sigma_C$. Note then that $T^+ \supseteq \Th_{\Lring}(K) \cup \Sigma_C \cup \UCd$.

For the next two proofs, we will at times need to refer to notions in both the sense of $T^+$ and the sense of $\DCF$. In the second case, we will always include this as a superscript; if no superscript is given, the notion should be understood in the sense of whatever model of $T^+$ we are working in. The full, field-theoretic algebraic closure of $A$ is denoted by $\tilde{A}$, and $A^\text{alg}$ denotes the relative, field-theoretic algebraic closure of $A$. If $A \supseteq C$, then $A^\text{alg}$ is equal to $\acl_{\L}(A)$ since models of $T^+$ are very $\L$-slim.

\begin{lemma}\label{lemma-pac-acl-is-compositum}
Let $(F, \partial, C) \models T^+$, and $A \leq B \leq F$ with $A$ $\acl$-closed in the sense of $T^+$. Then
\[
\acl^{\D\text{\normalfont -CF}_0}(B) = \acl(B) \cdot \acl^{\D\text{\normalfont -CF}_0}(A).
\]
\end{lemma}
\begin{proof}
Since $B$ is a $\D$-field containing $C$ and $F$ is very $\L$-slim, $\acl(B) = B^\text{alg}$, and $\acl^{\D\text{\normalfont -CF}_0}(B) = \tilde{B}$. So we need to show $\tilde{B} = B^\text{alg} \cdot \tilde{A}$. The proof of Proposition~4.6(2) of \cite{chatzidakis_pillay} shows that the restriction maps $\text{Gal}(F) \to \text{Gal}(A)$ and $\text{Gal}(F) \to \text{Gal}(\acl(B))$ are isomorphisms, and hence the restriction map $\text{Gal}(\acl(B)) \to \text{Gal}(A)$ is an isomorphism. Therefore, any automorphism of $\tilde{B}$ that fixes $B^\text{alg} \cdot \tilde{A}$ must also fix $\tilde{B}$. Since we are in characteristic $0$, $\tilde{B} / B^\text{alg} \cdot \tilde{A}$ is a Galois extension, and so $\tilde{B} = B^\text{alg} \cdot \tilde{A}$.
\end{proof}

\begin{theorem}\label{thm-simplicity-and-ei}
Let $(F, \partial, C) \models T^+$ and $(E, \partial, C) \subseteq (F, \partial, C)$. Then
\begin{enumerate}
    \item[\normalfont (1)] $\acl(E) = E^\text{\normalfont alg}$;
    \item[\normalfont (2)] if $E = \acl(E)$, then $T^+ \cup \diag(E)$ is complete;
    \item[\normalfont (3)] $T^+$ is model complete;
    \item[\normalfont (4)] the independence theorem holds for $T^+$ over algebraically closed sets;
    \item[\normalfont (5)] $T^+$ is simple and forking is given by forking independence in $\D\text{\normalfont -CF}_0$;
    \item[\normalfont (6)] $T^+$ has elimination of imaginaries.
\end{enumerate}
\end{theorem}
\begin{proof}
(1). By Lemma~\ref{structure-of-acl} since $F$ is very $\L$-slim (it is model complete, large, and characteristic $0$).

(2). By Proposition~4.6(2) of \cite{chatzidakis_pillay}, $\Th_{\Lring}(K) \cup \Sigma_C \cup \diag_{\L}(E)$ is complete. Then Theorem~\ref{uc-first-theorem} tells us that $\Th_{\Lring}(K) \cup \Sigma_C \cup \UCd \cup \diag(E)$ is complete. So $T^+ \cup \diag(E)$ is complete.

(3). By Theorem~\ref{UC-consequences}(i) since $\Th_{\Lring}(K) \cup \Sigma_C$ is model complete (Proposition~4.6(1) of \cite{chatzidakis_pillay}).

(4). This follows from the proof of Theorem~\ref{simplicity-transfers} and the fact that the independence theorem over algebraically closed sets holds for bounded PAC fields (Theorem~4.7 of \cite{chatzidakis_pillay}). Note that the proof of Theorem~\ref{simplicity-transfers} only uses the fact that $M$ is a model if the independence theorem for $T$ holds only for models.

(5). By Theorem~\ref{simplicity-transfers} and the corresponding result for bounded PAC fields (Corollary~4.8 of \cite{chatzidakis_pillay}) we know that $T^+$ is simple and forking independence is given by linear disjointness after closing under $\acl$---the relative algebraic closure of the $\D$-field it generates. We can then use general properties of linear disjointness of regular extensions, along with Lemma~\ref{lemma-pac-acl-is-compositum}, to show that, if $A$, $B$, and $D$ are all $\acl$-closed, then $A$ and $B$ are linearly disjoint over $D$ if and only if $\tilde{A}$ is linearly disjoint from $\tilde{B}$ over $\tilde{D}$. This is precisely forking independence in $\D\text{-CF}_0$ (see Theorem~5.9 of \cite{moosa_scanlon_2013}).

(6). This proof is essentially a combination of Theorem~5.12 of \cite{moosa_scanlon_2013} and Theorem~5.6 of \cite{hoffman-leon-sanchez}. Nonetheless, some details will be provided. We will assume that $(F, \partial, C)$ is a monster model of (some completion of) $T^+$, and that $(\mathfrak{D}, \partial)$ is a monster model of $\D\text{-CF}_0$ extending it. We write $\forkindep$ for nonforking independence in $(F, \partial, C)$. If we omit a superscript from an operator, we mean in the sense of $(F, \partial, C)$. 

We need the notion of dimension from Definition~5.10 of \cite{moosa_scanlon_2013}. If $K$ is a $\D$-field, then $\dim_\D(a/K) = (\trdeg(\nabla_r(a)/K) \colon r < \omega) \in \omega^\omega$, where $\nabla_r(a)$ is the tuple applying words of length at most $r$ on $\{ \partial_1, \ldots, \partial_l \}$ to $a$. Note that $\dim_\D(a/K) = \dim_\D(a/\tilde{K})$. Using Lemma~5.11 of \cite{moosa_scanlon_2013}, we then get that if $L/k$ is a regular extension, $\dim_\D(a/k) = \dim_\D(a/L)$ if and only if $\acl^{\D\text{-CF}_0}(ka)$ is linearly disjoint from $\tilde{L}$ over $\tilde{k}$ if and only if $a \forkindep[k] L$.

Let $e \in (F, \partial, C)^\text{eq}$ given by a $0$-definable function $f$ and a finite real tuple $a \in F$, that is, $f(a) = e$. Let $E = \acl^{\text{eq}}(e) \cap F$ and let $Q$ be the set of realisations of $\tp(a/E)$. Having established that $\dim_\D$ measures nonforking independence in $T^+$, the same proof as in Theorem~5.12 of \cite{moosa_scanlon_2013} allows us to find some $u \in Q$ such that $f(u) = e$ and $u \forkindep[E] a$.

We now follow the rest of the argument in Theorem~5.6 of \cite{hoffman-leon-sanchez}. Let $D = \{ d \in Q \colon f(d) = e \}$. If $D=Q$, then $e \in \dcl^\text{eq}(E)$ and we get weak elimination of imaginaries.

If $D \subsetneq Q$, let $d_0 \in Q \setminus D$ and $d \equiv_E d_0$ with $d \forkindep[E] D$. If $f(d) = e$, then $d \in D$ and hence $d \in \acl(E) = E$. So $d \in \acl^\text{eq}(e)$. Since $f(d)=e$, $e \in \dcl^\text{eq}(d)$, and we get weak elimination of imaginaries.

So assume $f(d) \not = e$. Now $u \equiv_E d$, $u \forkindep[E] a$, and $u \forkindep[E] d$. By the independence theorem over algebraically closed sets, we get $m \models \tp(u/Ea) \cup \tp(d/Eu)$ with $m \forkindep[E] au$. But this contradicts $f(d) \not = e$. Finally, since we are in a theory of fields and we have weak elimination of imaginaries, we have elimination of imaginaries.
\end{proof}

\section{The non-local case}\label{sec-non-local}

Recall that, throughout Sections~3--6, we assumed that either the $k$-algebra $\D$ was a local ring or each component in its local decomposition had residue field $k$. In this section we make some observations about the existence of model companions of $\D$-fields in the case when neither assumption holds. Without Assumption~\ref{res-field-k} the associated homomorphisms of a $\D$-field are not necessarily \emph{endo}morphisms, and hence it does not make sense to ask whether $T \cup \text{``$\D$-fields''}$ has a model companion when $T$ is a theory of difference fields. However, it does make sense to ask the question as $T$ varies over theories of fields. The main result of this section says that  when the base field $k$ is finitely generated over $\Q$, we get a full characterisation of when the uniform companion for large $\D$-fields exists: it exists if and only if $\D$ is local.

We start with the general case: $k$ is a field of characteristic $0$, $\D$ is a finite-dimensional $k$-algebra, and $\D = \prod_{i=0}^t B_i$ where each $B_i$ is a local, finite-dimensional $k$-algebra. We no longer impose Assumption~\ref{D-is-local} (that $\D$ is local) or even Assumption~\ref{res-field-k} (that the residue field of each $B_i$ is $k$). For $i>0$, the residue field of $B_i$ is $k[x]/(P_i)$ for some $k$-irreducible polynomial $P_i$, and that of $B_0$ is $k$. For an $\L$-theory $T$, the $\L(\partial)$-theory $T \cup \text{``$\D$-fields''}$ is denoted by $T_\D$, and the $\L(\sigma)$-theory $T \cup \text{``$\sigma$ is an endomorphism''}$ is denoted by $T_{\sigma}$.

A result of Kikyo and Shelah \cite{Kikyo2002TheSO} states that if a model complete theory has the strict order property, then the theory obtained by adding an automorphism has no model companion. In particular, if $\D = k \times k$, $\D$-fields correspond to fields with an endomorphism, and so $\RCF_\D = \RCF_\sigma$ and $\Th(\Q_p)_\D = \Th(\Q_p)_\sigma$ have no model companion. In fact, the Kikyo--Shelah theorem implies that $T_\D$ has no model companion when $\D$ is not local and $T$ has a model in which one of the polynomials $P_i$ has a root. We first prove this for the case when $\D$ has at least one local component with residue field $k$, and then reduce the more general statement to this case.

\begin{proposition}\label{D-mc-implies-sigma-mc}
Assume $\D$ is such that one of the local components $B_i$ has residue field $k$ for $i>0$. If $T_\D$ has a model companion, then $T_{\sigma}$ has a model companion.
\end{proposition}
\begin{proof}
Note that by a particular choice of the basis $\varepsilon_0, \ldots, \varepsilon_l$, we may assume that the associated endomorphism $\sigma_i$ corresponding to $B_i$ appears as one of the operators $\partial_j$. So $\L(\sigma) \subseteq \L(\partial)$.

Write $T^+$ for the model companion of $T_\D$ and $T^-$ for its reduct to $\L(\sigma)$. We will show that $T^-$ is the model companion of $T_{\sigma}$; clearly their universal parts coincide, so it suffices to prove $T^-$ is model complete.

Let $(K, \sigma) \models T^-$. We will show that $T^- \cup \diag_{\L(\sigma)}(K)$ is complete. Use Lemma~\ref{extending-D-structures} to equip $K$ with a $\D$-ring structure whose $i$th associated homomorphism is $\sigma$ and whose $j$th associated homomorphism is inclusion $K \to K[x]/(P_j)$ for $j \not = i$. Then $K \models T_\D$, and it embeds in some $L \models T^+$. Since $T^+$ is model complete, $T^+ \cup \diag_{\L(\partial)}(L)$ is complete, and hence its reduct to $\L(\sigma)(K)$, $T^- \cup \diag_{\L(\sigma)}(K)$, is complete.
\end{proof}

We now weaken the assumption that the residue field of some $B_i$ is $k$ to the assumption that $T$ has a model $K$ in which one of the polynomials $P_i$ has a root. If $T_\D$ has a model companion, then $T_\D \cup \diag(K)$ has a model companion. Let $\mathcal{E}$ be the $K$-algebra $\D \otimes_k K$. As mentioned in the proof of Theorem~3.2 of \cite{ozlem_2018}, if $L$ is an $\mathcal{E}$-field, then $\mathcal{E}$-field extensions of $L$ coincide with $\D$-field extensions of $L$. Hence if $T_\D \cup \diag(K)$ has a model companion, so does $T_{\mathcal{E}}$. But $\mathcal{E}$ now satisfies the assumption in Proposition~\ref{D-mc-implies-sigma-mc}. So we have proved the following.

\begin{corollary}\label{thm-modifies-kikyo-shelah}
If $T$ is model complete and has a model with the strict order property in which one of the $P_i$ has a root, then $T_\D$ has no model companion.
\end{corollary}

Real closed fields and $\Q_p$ have the strict order property, and so this result means if any $P_i$ has a root in some real closed field or some $p$-adically closed field, there is no uniform companion. In particular, if the base field $k$ is a finitely generated field extension of $\Q$, we get a full converse to the main theorem.

\begin{corollary}
Suppose $k$ is a finitely generated field extension of $\Q$. Then there is a uniform companion for theories of large $\D$-fields if and only if $\D$ is a local ring.
\end{corollary}
\begin{proof}
If $\D$ is local, $\D$-fields whose associated difference field is difference large correspond precisely to $\D$-fields whose underlying field is large. The uniform companion then exists by Section~\ref{sec-uc}.

Suppose $\D$ is not local. Then the splitting field of the polynomial $P_1 \in k[x]$ is a finitely generated extension of $\Q$, and so by Theorem~1 of \cite{cassels_1976}, embeds in some $\Q_p$. Hence $P_1$ has a root in $\Q_p$. Then by Theorem~\ref{thm-modifies-kikyo-shelah}, $\Th(\Q_p)_\D$ has no model companion.
\end{proof}

\begin{remark}
The base field $k$ does have an impact on when the uniform companion exists. If $k$ is algebraically closed, then the only model complete theory of fields containing $k$ is $\text{ACF}_0$, and hence the existence of a uniform companion for $\D$-fields is equivalent to the existence of the model companion of $\text{ACF}_0 \cup \text{``$\D$-fields''}$; this exists for all $\D$ by Theorem~7.6 of \cite{moosa_scanlon_2013}.

However, for other fields $k$ the question is still open. For instance, suppose $k = \R$. No model of $\Th(\Q_p)$ can be an $\R$-algebra, and so $\Th(\Q_p) \cup \text{``$\D$-fields''}$ is inconsistent. Hence the above method does not show that there is no uniform companion in the case $\D = \R \times \C$, for instance.
\end{remark}

\end{document}